\documentclass[
    10pt,
    english,
    a4paper,
    final,
]{article}
\usepackage{amsmath}
\usepackage{amssymb}
\usepackage{amsthm}
\usepackage{accents}
\usepackage[english]{babel}
\usepackage{bbm}
\usepackage{chngcntr}
\usepackage{enumitem}
\usepackage{listings}
\usepackage{lmodern} 
\usepackage{mathtools}
\usepackage[unicode=true]{hyperref}
\usepackage{setspace,graphicx,tikz,tabularx} 
\usepackage{parskip} 
\usepackage{csquotes}
\usepackage[textwidth=16cm,textheight=23.5cm]{geometry}
\usepackage[textwidth=20mm,textsize=footnotesize,color=orange!50]{todonotes}
\usepackage[capitalise,noabbrev,nameinlink,sort]{cleveref}
\let\etoolboxforlistloop\forlistloop 
\usepackage{autonum}
\let\forlistloop\etoolboxforlistloop 
\cslet{blx@noerroretextools}\empty 
\usepackage[style=numeric,url=false,sortcites=true,isbn=false,doi=false,eprint=true,date=year,maxbibnames=9]{biblatex}
\renewbibmacro{in:}{}

\setuptodonotes{noline}

\DeclareNameAlias{default}{family-given}
\counterwithin{equation}{section}

\DeclarePairedDelimiter\abs{\lvert}{\rvert}%
\DeclarePairedDelimiter\norm{\lVert}{\rVert}%

\let\predefinedC\c
\renewcommand{\c}[1]{\mathcal{#1}}

\renewcommand{\b}[1]{\mathbb{#1}}

\newcommand{\N}{\mathbb{N}}
\newcommand{\R}{\mathbb{R}}
\renewcommand{\P}{\mathbb{P}}
\DeclareMathOperator{\E}{\mathbb{E}}
\DeclareMathOperator{\Var}{Var}

\newcommand{\1}{\mathbbm{1}}

\newtheorem{theorem}{Theorem}[section]
\newtheorem{lemma}[theorem]{Lemma}
\newtheorem{remark}[theorem]{Remark}

\newtheorem{proposition}[theorem]{Proposition}
\newtheorem{definition}[theorem]{Definition}

\newtheorem{assumption}{Assumption}
\crefname{assumption}{Assumption}{Assumptions}

\renewcommand{\theassumption}{\Alph{assumption}}
\newlist{assumptionenum}{enumerate}{1} 
\setlist[assumptionenum]{label=(\theassumption\arabic*)}
\crefalias{assumptionenumi}{assumption}

\makeatletter
\DeclareRobustCommand{\crefcompress}[1]{%
  \begingroup\@cref@compresstrue\cref{#1}\endgroup
}
\makeatother

\let\eqref\labelcref
\makeatletter
\autonum@generatePatchedReferenceCSL{eqref}
\makeatother

\begin{document}

\title{Extended mean-field games with multi-dimensional singular controls and non-linear jump impact}

    \author{Robert Denkert\thanks{Humboldt University Berlin, Department of Mathematics, Unter den Linden 6, 10099 Berlin. Denkert gratefully acknowledges financial support through the International Research Training Group (IRTG) 2544 \textsl{Stochastic Analysis in Interaction}.} \quad \quad Ulrich Horst\thanks{Humboldt University Berlin, Department of Mathematics and School of Business and Economics, Unter den Linden 6, 10099 Berlin. Horst gratefully acknowledges support from DFG CRC/TRR 388 ``Rough Analysis, Stochastic Dynamics and Related Fields", Project B04. }}
    
    \maketitle

\vspace{-2.6mm}
\begin{abstract}
We establish a probabilistic framework for analysing extended mean-field games with multi-dimensional singular controls and state-dependent jump dynamics and costs. Two key challenges arise when analysing such games: the state dynamics may not depend continuously on the control and the reward function may not be u.s.c.~Both problems can be overcome by restricting the set of admissible singular controls to controls that can be approximated by continuous ones. We prove that the corresponding set of admissible weak controls is given by the weak solutions to a Marcus-type SDE and provide an explicit characterisation of the reward function. The reward function will in general only be u.s.c.~To address the lack of continuity we introduce a novel class of MFGs with a broader set of admissible controls, called MFGs of parametrisations. Parametrisations are laws of state/control processes that continuously interpolate jumps. We prove that the reward functional is continuous on the set of parametrisations, establish the existence of equilibria in MFGs of parametrisations, and show that the set of Nash equilibria in MFGs of parametrisations and in the underlying MFG with singular controls coincide. This shows that MFGs of parametrisations provide a canonical framework for analysing MFGs with singular controls and non-linear jump impact.    
\end{abstract}
    
    


\section{Introduction}

Mean-field games (MFGs) are a powerful tool to analyse strategic interactions in large populations when each individual player has only a small impact on the behaviour of other players. Introduced independently by Huang et al \autocite{caines_large_2006} and Lasry and Lions \autocite{lasry_mean_2007}  MFGs have been successfully applied to many problems, ranging from banking networks and models of systemic risk \autocite{carmona_mean_2015}, to dynamic contracting problems \autocite{elie_tale_2019}, bitcoin mining \autocite{li_mining_2023} and the mitigation of epidemics \autocite{aurell_epidemics_2022}, and from problems of optimal trading under market impact \autocite{carmona_probabilistic_2015,huang_mean-field_2019,fu_mean_2021,fu_mean-field_2023,fu_meanfield_2024,cardaliaguet_mean_2018}, to models of optimal exploitation of exhaustible resources \autocite{chan_fracking_2017,graewe_maximum_2022}, economic growth \autocite{gozzi_growth_2022} and energy production \autocite{aid_entry_2021,dumitrescu_energy_2024,elie_energy_2021}.

In a standard $N$-player game, each player $i\in \{1,\dotso,N \}$ chooses an action $u_i$ to maximise an individual reward of the form
\begin{equation}\label{cost-fun}
    J^i(u) = \E\Big[ g(\bar{\mu}^N_T,X^i_T) + \int_0^T f(t,\bar{\mu}^N_t,X^i_t,u^i_t)dt \Big],
\end{equation}
subject to the $d$-dimensional state dynamics
\begin{equation}\label{state-without-singular}
    dX^i_t = b(t,\bar{\mu}^N_t,X^i_t,u^i_t) dt + \sigma(t,\bar{\mu}^N_t,X^i_t,u^i_t) dW^i_t,\quad X^i_0 = x_0.
\end{equation}
Here $W_1,\dotso,W_N$ are independent $m$-dimensional Brownian motions on some underlying filtered probability space, $u=(u_1,\dotso,u_N)$ the combined action of all players, each action $u^i=(u^i_t)_{t \in [0,T]}$ of player $i$ an adapted stochastic process, and $\bar\mu^N_t \coloneqq \frac 1 N \sum_{j=1}^N \delta_{X^j_t}$ the empirical distribution of the players' states at time $t\in [0,T]$.

In view of the diminishing impact of an individual player's action on other players' choices in large populations the existence of approximate Nash equilibria for large populations can be established using a representative agent approach. The idea is to first consider the optimisation problem of a representative agent where the empirical distribution of the players' states is replaced by an external measure flow $\tilde\mu$, and then to solve the fixed point problem $\tilde\mu = \c L(X^{\tilde\mu,*})$ where $X^{\tilde\mu,*}$ denotes the state process of the optimal strategy for the representative player given the measure flow $\tilde\mu$.\footnote{The idea of decoupling the macroscopic from the microscopic dynamics has also been applied to models of social interactions in e.g. \autocite{horst_equilibria_2006}}
Using the representative agent approach, a MFG can be formally described as follows:
\begin{align}\label{standard-MFG}
\begin{cases}
    1.&\text{fix a measure flow }[0,T]\ni t\mapsto \tilde\mu_t \in \c P_2(\R^d),\\
    2.&\text{solve the stochastic optimisation problem}\\
   & \sup_u \E\Big[g(\tilde{\mu}_T,X_T) + \int_0^T f(t,\tilde{\mu}_t,X_t,u_t)dt\Big] \\
   & \text{subject to the state dynamics} \\
    &dX_t = b(t,\tilde\mu_t,X_t,u_t) dt + \sigma(t,\tilde\mu_t,X_t,u_t) dW_t, ~ X_{0} = x_{0}.\\
    3.&\text{solve the fixed point problem }\c L(X^{\tilde \mu, *}) = \tilde\mu
\end{cases}
\end{align}
where $\c P_2(\R^d)$ denotes the space of all square-integrable probability measures on $\R^d$ equipped with the Wasserstein topology, $X^{\tilde \mu, *}$ denotes the state dynamics under an optimal policy $u^*$ and $\c L(X)$ denotes the law of a stochastic process $X$. 

The method to solve standard MFGs proposed in the initial paper by Lasry and Lions \autocite{lasry_mean_2007} is an analytical one where the the fixed point is characterised by a coupled forward-backward PDE system. The forward component is a Kolmogorov-Fokker-Planck equation describing the dynamics of the state process and the backwards component is the Hamilton-Jacobi-Bellman equation arising from the optimisation problem of the representative agent.
In \autocite{carmona_probabilistic_2013} Carmona and Delarue introduced a  probabilistic approach, where the fixed point is characterised in terms of a McKean-Vlasov forward-backward SDE (FBSDE) arising from a Pontryagin-type maximum principle.

A nonconstructive approach has been introduced by Lacker in \autocite{lacker_mean_2015} using a relaxed solution concept. The idea is to use the continuity of the reward functional of the representative agent in both his actions and the empirical distribution together with Berge's maximum principle to show that their best response map to a given measure flow $\tilde\mu$ is upper hemi-continuous. The existence of a Nash equilibria is established using the Kakutani-Fan-Glicksberg fixed point theorem.

An extension to MFGs with common noise has been considered by, e.g.~Carmona et al \autocite{carmona_mean_2016}. In the common noise case, the fixed point is random. This prevents an application of standard fixed point results. To overcome the problem of randomness, Carmona et al \autocite{carmona_mean_2016} introduced a notion of weak solutions to MFGs. Weak solutions are probability measures on path spaces that specify the distribution of the state and control processes.  

The existence of relaxed solutions to MFGs, respectively extended MFGs with singular controls was first established in \autocite{fu_mean_2017} and \autocite{fu_extended_2022}, respectively; MFGs with singular controls of bounded velocity were considered in \autocite{cao_approximation_2023}; MFGs with singular controls of finite fuel were considered in \autocite{guo_stochastic_2019}. MFGs with singular controls and their respective $N$-player approximations were studied in, e.g.\@ \autocites{campi_mean-field_2022,cao_approximation_2023,cao_mfgs_2022,guo_stochastic_2019,dianetti_ergodic_2023,cao_stationary_2022}. 

Extending the models considered in \autocite{fu_mean_2017,fu_extended_2022} our goal is to establish an existence of solutions result for extended MFGs with multi-dimensional singular controls $\xi \in \R^l$ and state/control-dependent jump dynamics and costs. This suggests to consider MFGs of the form 
\begin{align}\label{eq ill-behaved mean-field game singular control}
\begin{cases}
    1.&\text{fix a measure flow }[0,T]\ni t\mapsto \tilde\mu_t \in \c P_2(\R^d\times\R^l),\\
    2.&\text{solve the stochastic optimisation problem}\\
       &\sup_\xi \E \Big[g(\tilde\mu_T,X_T,\xi_T) + \int_0^T f(t,\tilde\mu_t,X_t,\xi_t) dt - \int_0^T c(t,X_t,\xi_t) d\xi_t\Big],\\
    &\text{over all non-decreasing, càdlàg processes $\xi:[0,T] \to \R^l$ s.t. the state dynamics}\\
    &dX_t = b(t,\tilde\mu_t,X_t,\xi_t) dt + \sigma(t,\tilde\mu_t,X_t,\xi_t) dW_t + \gamma(t,X_t,\xi_t) d\xi_t,\quad X_{0-} = x_{0-},\\
    3.&\text{solve the fixed point problem }\c L(X,\xi) = \tilde\mu.
\end{cases}
\end{align}
The above MFG is well-defined if the control $\xi$ is continuous. However - for reasons outlined in the next two paragraphs - the above MFG does not provide a good framework for analysing game-theoretic models with singular controls and non-linear jump impact.


Singular controls should be understood as limits of continuous (``regular'') controls or of controls with many small jumps, which are not continuous operations in the standard Skorokhod J1-topology on the space of c\`adl\`ag functions. This calls for the use of weaker topologies when working with singular controls. Unfortunately, new problems arise when working with weaker topologies. 

For instance, if the function $\gamma$ depends on the state or control variable, then the state process may not depend continuously on the control in the weak M1-topology.\footnote{For instance, $\xi^n\coloneqq 0\lor n(\cdot -1+1/n)\land 1 \to \1_{[1,2]}\eqqcolon\xi$ in the M1-topology while $\int_0^2 \xi^n_{t-} d\xi^n_t = \frac 1 2 \not\to \int_0^2 \xi_{t-} d\xi_t = 0$.} Likewise, if the cost term $c$ depends on the state or control variable, then the reward functional 
\[
    J(\tilde \mu, \xi) \coloneqq \E \Big[g(\tilde\mu_T,X_T,\xi_T) + \int_0^T f(t,\tilde\mu_t,X_t,\xi_t) dt - \int_0^T c(t,X_t,\xi_t) d\xi_t\Big]
\]
may not be upper semi-continuous in the control variable in the J2-topology. Specifically,  so-called ``chattering strategies'' may emerge where approximating a singular control by a sequence of controls with many small jumps may generate lower costs and hence higher rewards than employing the singular control itself; this effect has first been observed in \autocite{alvarez2000singular}. 

We show that the problems outlined above can be overcome by first defining the state dynamics and reward functionals for continuous controls and then restricting the set of admissible singular controls to those controls that can be approximated by continuous ones with respect to the weak M1-topology, a topology that has been successfully used in the context of singular control by many authors, including \autocite{cohen2021singular,fu_mean_2017,fu_extended_2022,denkert2023extended}.
Although the weaker Meyer-Zheng topology also has been utilised in, e.g.~\autocite{li_existence_2017, dianetti_nonzero-sum_2020}, we prefer to work with the $WM_1$ topology. The topology is weak enough to also have the properties that make the Meyer-Zheng topology appealing when working with singular controls, namely (i) being able to approximate jumps continuously, and (ii) having weak compactness conditions such as every set of bounded monotone functions being compact. At the same time, unlike the Meyer-Zheng topology, the $WM_1$ topology admits for an explicit and convenient representation of a metric via parametrisations.

In a first step, we establish an intuitive representation of our set of admissible weak controls. Specifically, we prove that the set of admissible (weak) controls corresponds to the weak solutions to a Marcus-type SDE. For continuous controls the Marcus-type SDE reduces to a standard SDE; for singular controls the idea is to smooth the discontinuities in the state process resulting from the singularities in the control process. 

In a second step, we provide an explicit representation of the corresponding reward function in terms of minimal jumps costs as in \autocite{denkert2023extended}. The approach extends the case of one-dimensional controls studied in, e.g.\@ \autocite{de_angelis_stochastic_2018,dufour_singular_2004} to multi-dimensional settings.
In \autocite{fu_extended_2022}, the author addresses singular mean-field games with multi-dimensional controls under the additional assumption that the dimensions contribute independently to the reward functional. This condition is not needed in this work. We represent the reward function using parametrisations of the state/control process. Parametrisations are continuous state/control processes running on exogenous time-scales; they have been successfully applied in \autocite{denkert2023extended} to solve a broad class of mean-field control problems with multi-dimensional singular controls.

Our choice of admissible strategies avoids chattering strategies but the resulting reward function may still only be u.s.c. To overcome this problem we introduce in a third step MFGs of parametrisations where the set of admissible controls is given by the set of parametrisations of the state/control process, that is, by distributions of state-control processes running on different time scales. The new time scales allow for a continuous interpolations of the discontinuities in the state/control process. 

The main advantage of working with parametrisations is that the reward function is continuous in parametrisations. Establishing the existence of equilibria in MFGs of parametrisations is hence not difficult; it can easily be derived from the well understood regular control case. 

Since every singular control gives rise to a parametrisation, MFGs of parametrisations allow for a larger class of admissible strategies, hence equilibria are more likely to exist. In a fourth and final step we, therefore, prove that the set of Nash equilibria in both games coincide under a natural condition on the jump coefficient $\gamma$. Under this assumption we show that any parametrisation gives rise to an admissible singular control and that any Nash equilibrium in a MFG of parametrisations induces an equilibrium (in weak form) in the original MFG in singular controls, and vice versa. This shows that MFGs of parametrisations provide a unified and transparent approach to solve MFGs with singular controls and non-linear jump dynamics.%
\footnote{We emphasise that our focus is on the existence of equilibria; we do not address the problem of uniqueness of equilibria.}  

When we can derive the existence of Nash equilibria for both the singular control MFG and the MFG of parametrisations by constructing Nash equlibria of a suitable sequence of approximating regular control MFGs, we only show compactness of this sequence. This approach does not rule out that there exist multiple limits, each possibly being a equilibrium to our original MFG. Moreover, we cannot guarantee that every Nash equilibrium can be obtained this approximation approach in the first place.

The remainder of this paper is organised as follows. Our MFG will be formally introduced in \cref{section derivation singular mfg} where we also state our main results and assumptions. \cref{section parametrisations} introduces parametrisations of singular controls and derives a transparent representation of the reward function. MFGs of parametrisations are introduced and solved in \cref{section mean-field game parametrisations} where we also establish the existence of equilibria in MFGs with singular controls. Section \ref{conclusion} concludes.

\textbf{Notation.} Given a probability measure $\mu$ and a random variable $X$, we denote by $\mu_X \coloneqq \mu\# X \coloneqq \mu\circ X$ the push-forward of $X$ under $\mu$. In particular, if $\mu$ denotes the distribution of a random vector $(X,Y)$, then $\mu_X$ denotes the law of $X$. By $\E^\mu$ we denote the expectation w.r.t.~a measure $\mu$. By $\xi:[0,T] \to \R^l$ we denote an $l$-dimensional càdlàg function that is non-decreasing in each component (``control''). We denote the space of continuous (respectively càdlàg) functions from $[0,T]$ to $\R^l$ by $C([0,T];\R^l)$ (respectively $D([0,T];\R^l)$). For a Polish space $(E,\c B(E))$ by $\c P_2(E)$ the space of all square-integrable probability measures equipped with the 2-Wasserstein topology. We denote by $\abs{\cdot}$ the euclidean norm.


\section{Our mean-field game}\label{section derivation singular mfg}

In this section we introduce our approach to solve MFGs with singular controls and state-dependent jump dynamics and jump costs. We employ the solution concept of weak solutions where the set of admissible controls is given by a set of probability measures on the canonical path space. To allow for jumps at the initial time we fix some $\varepsilon>0$ and define our path space as
\[
    D^0 = D^0([0,T];\R^d\times\R^l) \subseteq D([-\varepsilon,T];\R^d\times\R^l),
\]
as the subset of all càdlàg functions that are constant on $[-\varepsilon,0)$; the particular choice of $\varepsilon$ is not important for our results. We equip the set $D^0$ with the weak M1 ($WM_1$) topology.\footnote{The weak $M1$-topology will be formally introduced in \cref{definition wm1 parametrisation whitt} below.} 

To motivate our choice of action space and reward function we do not introduce the MFG right away but proceed step-by-step instead. We first consider the continuous control problem and then derive a ``canonical formulation'' for the resulting singular control problem based on approximations of singular controls by continuous ones. Subsequently, we derive an explicit representation of the resulting reward functional in terms of minimal jump costs. With these auxiliary results in hand, we finally introduce our MFG in singular controls.


\subsection{Continuous controls}

Singular controls are generally derived from, and should be understood as the limit of continuous (regular) controls.
We hence choose an approach based on continuous approximations of singular controls and first introduce weak solutions to the \emph{continuous} control problem. 

\begin{definition}
    We consider the canonical state space
    \[
    \tilde\Omega \coloneqq D^0([0,T];\R^d\times\R^l\times\R^m),
    \]
    equipped with the Borel $\sigma$-algebra, the canonical process $(X,\xi,W)$ and the canonical filtration $\tilde{\b F} \coloneqq \b F^{X,\xi,W}$. Given a measure flow $\tilde\mu\in\c P_2(D^0)$, we call a probability measure $\P\in\c P_2(\tilde\Omega)$ a weak solution to the SDE
    \begin{align}
        dX_t &= b(t,\tilde\mu_t,X_{t-},\xi_{t-}) dt + \sigma(t,\tilde\mu_t,X_{t-},\xi_{t-}) dW_t + \gamma(t,X_{t-},\xi_{t-}) d\xi_t,\qquad t\in [0-,T],\quad X_{0-} = x_{0-},
        \label{eq continuous singular control state dynamics}
    \end{align}
    if the dynamics is $\P$-a.s.\@ satisfied and if $W$ is an $m$-dimensional $\tilde{\b F}$-Brownian motion.
\end{definition}

We now define the set of continuous admissible (weak) controls as the set of all probability measures on the canonical path space with finite second moment that correspond to continuous weak solutions of the SDE \eqref{eq continuous singular control state dynamics}. The associated reward is defined as the expected payoff under that control.

\begin{definition}\label{definition continuous admissible controls}
    Given a measure flow $\tilde\mu\in\c P_2(D^0)$, we define the set of continuous admissible controls $\c C(\tilde\mu)$ as the set of all probability measures $\mu\in \c P_2(D^0)$ such that
    \begin{enumerate}[label=(\roman*)]
        \item $\xi$ is continuous and non-decreasing in every component $\mu$-a.s.,
        \item $\xi_{0-} = 0$, $\mu$-a.s.,
        \item $\mu$ is the marginal law of a weak solution $\P$ to the SDE \eqref{eq continuous singular control state dynamics}.
    \end{enumerate}
    Furthermore, we associate the following reward to a continuous admissible control $\mu\in\c C(\tilde\mu)$:
    \[
    J_{\c C}(\tilde\mu,\mu) \coloneqq \E^\mu \Big[g(\tilde\mu_T,X_T,\xi_T) + \int_0^T f(t,\tilde\mu_t,X_t,\xi_t) dt - \int_{[0,T]} c(t,X_t,\xi_t) d\xi_t\Big].
    \]
\end{definition}


\subsection{Singular controls and Marcus-type SDEs}

Our key motivation is to identify the set of admissible singular controls with the closure $\overline{\c C(\tilde\mu)}$ of the set of continuous controls in $D^0$ w.r.t.\@ the $WM_1$ topology. In other words, we assume that a singular control is admissible iff it can be approximated by continuous ones, and extend the reward functional $J_{\c C}(\cdot)$ to the set $\overline{\c C(\tilde\mu)}$ as follows:
\begin{align}\label{eq J singular controls as limit}
    J(\tilde\mu,\mu) \coloneqq \limsup_{\substack{(\mu^n)_n\subseteq \c C(\tilde\mu)\\\mu^n\to\mu\text{ in }D^0}} J_{\c C}(\tilde\mu,\mu^n), \quad \mu \in \overline{\c C(\tilde\mu)}.
\end{align}

In principle we could now formulate a MFG in terms of the abstract action space $\overline{\c C(\tilde\mu)}$ and the abstract u.s.c.~reward function $J$. By construction chattering strategies cannot emerge in this game. However, this setting would be too inconvenient to work with. 

Instead, we prove that under mild technical assumptions the set $\overline{\c C(\tilde\mu)}$ can be represented as the set of weak solutions to the Marcus-type SDE
\begin{align}
    dX_t &= b(t,\tilde\mu_t,X_{t-},\xi_{t-}) dt + \sigma(t,\tilde\mu_t,X_{t-},\xi_{t-}) dW_t + \gamma(t,X_{t-},\xi_{t-}) \diamond d\xi_t,\quad t\in [0-,T],\quad X_{0-} = x_{0-}.
    \label{eq singular control state dynamics}
\end{align}
The Marcus-type integration operator is defined by
\[
\gamma(t,X_{t-},\xi_{t-}) \diamond d\xi_t \coloneqq \gamma(t,X_{t-},\xi_{t-}) d\xi^c_t + \1_{\{\xi_{t-}\not= \xi_t\}} (\psi(t,X_{t-},\xi_{t-},\xi_t) - X_{t-}),
\]
where $\xi^c$ denotes the continuous part of the control $\xi$, and $\psi(t,x,\xi,\xi') = y_1$ where $y:[0,1] \to \R^d$ is the unique solution to the ODE
\begin{align}\label{eq marcus integration jump}
    dy_u &= \gamma(t,y_u,\zeta_u) d\zeta_u,\quad y_0 = x,
\end{align}
with $\zeta_u = (1-u)\xi + u \xi'$ being the linear interpolation from $\xi$ to $\xi'$. The idea of the Marcus-type dynamics is to continuously interpolate discontinuities in the state process through \emph{linear} interpolation of the jumps in the control variable. 
\begin{remark}
Under standard Lipschitz assumptions the ODE \eqref{eq marcus integration jump} has a unique solution for any non-decreasing, continuous path $\zeta$. If $\gamma(t,x,\xi) = \gamma(t)$, or if $\xi$ is continuous, then 
\[ 
    \gamma(t,X_{t-},\xi_{t-}) \diamond d\xi_t = \gamma(t,X_{t-},\xi_{t-}) d\xi_t
\]
and the Marcus-type SDE reduces to a standard SDE. It is the dependence of the jump term $\gamma$ on the state and/or control variable that requires a non-standard setting when analysing the state dynamics. 
\end{remark}

\begin{remark}
In the case that $d=1$, we can write \eqref{eq marcus integration jump} as a standard ODE,
\[
y'(u) = \gamma(t,y(u),(1-u)\xi + u\xi') (\xi'-\xi),\quad y(0) = x.
\]
By solving this class of non-linear ODEs, we can explicitly describe the behaviour of the system under the Marcus-type dynamics given in \eqref{eq singular control state dynamics}.
\end{remark}

By analogy to the weak solutions approach to continuous controls introduced above, we now define the set of (weak) singular controls as the set of all probability measures on the same canonical path space $(\tilde\Omega,\tilde{\c F},\tilde{\b F})$ that coincide with weak solutions of our Marcus-type SDE.

\begin{definition}\label{definition weak solution sde admissible control}
    Given a measure flow $\tilde\mu\in\c P_2(D^0)$, we call a probability measure $\P\in\c P_2(\tilde\Omega)$ a weak solution to the Marcus-type SDE \eqref{eq singular control state dynamics}
    if the dynamics are $\P$-a.s.\@ satisfied and $W$ is an $\tilde{\b F}$-Brownian motion.
    Furthermore, we define the set of admissible (weak) singular controls $\c A(\tilde\mu)$ as the set of all $\mu\in \c P_2(D^0)$ such that
    \begin{enumerate}[label=(\roman*)]
        \item $\xi$ is non-decreasing $\mu$-a.s.,
        \item $\xi_{0-} = 0$, $\mu$-a.s.,
        \item $\mu$ is the marginal law of a weak solution $\P$ to the SDE \eqref{eq singular control state dynamics}.
    \end{enumerate}
    We call a control $\mu$ Lipschitz continuous with Lipschitz constant $L$ if $\xi$ is $\mu$-a.s.\@ $L$-Lipschitz.
\end{definition}

We now introduce two assumptions that guarantee that the set of admissible controls coincides with the weak solutions to our Marcus-type SDE, i.e.\@ that $\overline{\c C(\tilde\mu)} = \c A(\tilde\mu)$. The first assumption reads as follows.

\begin{assumption}\label{assumptions state dynamics}
The coefficients
\[
b:[0,T]\times \c P_2(D^0)\times\R^d\times\R^l\to\R^d,\quad\sigma:[0,T]\times \c P_2(D^0)\times\R^d\times\R^l\to\R^{d\times m}\quad\text{and}\quad\gamma:[0,T]\times\R^d\times\R^l\to\R^{d\times l}
\]
satisfy the following conditions.
\begin{assumptionenum}
    \item\label{assumption b sigma continuous} $b,\sigma$ are continuous in $t$ and Lipschitz continuous in $m,x,\xi$ uniformly in $t$.
    \item\label{assumptions gamma bounded continuous} $\gamma$ is globally bounded and continuous in $t$ and Lipschitz continuous in $x,\xi$ uniformly in $t$.
    \item \label{assumptions gamma montone}
    For all $t\in [0,T]$, $x\in\R^d$ and every non-decreasing, continuous path $\xi:[0,1]\to\R^l$, the function given by 
    \[
    dy_u = \gamma(t,y_u,\xi_u) d\xi_u,\quad y_0 = x,
    \]
    is monotone (either increasing or decreasing) in each component. 
\end{assumptionenum}
\end{assumption}

\cref{assumption b sigma continuous,assumptions gamma bounded continuous} are standard. \cref{assumptions gamma montone} guarantees that only monotone interpolations of the jumps in the state process are allowed. The assumption is satisfied if, for instance, $\gamma$ is non-negative or non-positive. 

Our following second assumption ensures that the Marcus-type integration $\gamma(t,X_{t-},\xi_{t-}) \diamond d\xi_t$ aligns with the approximation of càdlàg paths by continuous paths in the $WM_1$-topology. The Marcus-type integration assumes a linear interpolation of the jumps in the control variable while we allow for arbitrary continuous approximation of the control variable. Although different approximations of the control variable may well generate different costs, they should not result in different state dynamics.

\begin{assumption}
\label{assumptions gamma path independent}
The integration $\gamma(t,x,\xi)\diamond d\xi$ is path-independent, by which we mean that in \eqref{eq marcus integration jump}, replacing $u\mapsto (1-u)\xi + u\xi'$ by any other continuous, non-decreasing path $u\mapsto\zeta_u$ with $\zeta_0 = \xi$ and $\zeta_1 = \xi'$ leads to the same $y_1$.
\end{assumption}


The above assumption is satisfied if $\xi$ is one-dimensional; this case has been extensively studied in the singular control literature. In the one dimensional case there is only way to interpolate up to a reparametrisation of time and the choice of interpolations is not relevant. The situation is very different in the multi-dimensional case. If $\gamma$ is independent of $(x,\xi)$ as in  \autocite{fu_mean_2017,fu_extended_2022}  the condition holds as well. If $\gamma$ is independent of the state variable, i.e.\@ if $\gamma(t,x,\xi) = \gamma(t,\xi)$, the path-independence turns out to be equivalent to the vector field $\gamma(t,\cdot)$ being conservative for each $t\in [0,T]$.

Under the above assumptions the set of admissible controls coincides with the set of weak solutions to our Marcus-type SDE as shown by the following result. The proof  follows from \cref{lemma reducing parametrisations to controls,lemma wm1 convergence parametrisation convergence,lemma parametrisation construct lipschitz approximations} given below.

\begin{proposition}
    Under \cref{assumptions gamma path independent,assumptions state dynamics}, for all $\tilde\mu\in \c P_2(D^0)$, it holds that $\overline{\c C(\tilde\mu)} = \c A(\tilde\mu)$.
\end{proposition}

\subsection{The MFG}

Having established the ``natural'' state dynamics for singular controls with state-dependent jump dynamics and jumps costs we now establish a more transparent and canonical representation of our reward functional. Using an approach introduced in  \autocite{denkert2023extended} we show in \cref{lemma J control J sup parametrisations} below that the reward functional admits the following canonical representation
\begin{equation} \label{J1}
    J(\tilde \mu, \mu)
    = \E^\mu \Big[g(\tilde\mu_T,X_T,\xi_T) + \int_0^T f(t,\tilde\mu_t,X_t,\xi_t) dt - \int_0^T c(t,X_t,\xi_t) d\xi_t^c
    - \sum_{t\in [0,T]} C_{D^0}(t,X_{t-},\xi_{t-},\xi_t) \Big],
\end{equation}
for all $\mu\in\c A(\tilde\mu)$ where the function $C_{D^0}$ describes minimal jump costs. It is defined as follows.\footnote{The proof of \autocite[Lemma 3.2]{denkert2023extended} shows that under \cref{assumptions state dynamics,assumptions reward and cost functions} the minimum in \eqref{C_D} is indeed attained. }

\begin{definition}[{see also \autocite[Definition 3.1]{denkert2023extended}}]
    \label{definition C_D0}
    For $t\in [0,T]$, $x\in\R^d$ and $\xi,\xi'\in\R^l$ such that $\xi\leq \xi'$ component-wise, we define
    \begin{equation} \label{C_D}
    C_{D^0} (t,x,\xi,\xi') \coloneqq \min_{(y,\zeta)\in\Xi(t,x,\xi,\xi')} \int_0^1 c(t,y_u,\zeta_u) d\zeta_u,
    \end{equation}
    over the set $\Xi(t,x,\xi,\xi')$ of all non-decreasing, continuous paths $(y,\zeta):[0,1]\to \R^d\times\R^l$ such that $y_0 = x$, $\zeta_0 = \xi$, $\zeta_1 = \xi'$ and
    \[
    dy_u = \gamma(t,y_u,\zeta_u) d\zeta_u,\qquad u\in [0,1].
    \]
\end{definition}

With an explicit representation of both the set of admissible controls and the reward function in hand, we are now ready to introduce the following canonical MFG for singular controls.
\begin{align}
\begin{cases}
    1.&\text{fix a measure flow }[0,T]\ni t\mapsto \tilde\mu_t \in \c P_2(\R^d\times\R^l),\\
    2.&\text{solve the weak stochastic optimisation problem}\\
     & \sup_{\mu \in \c A(\tilde\mu)} \E^\mu \Big[g(\tilde\mu_T,X_T,\xi_T) + \int_0^T f(t,\tilde\mu_t,X_t,\xi_t) dt - \int_0^T c(t,X_t,\xi_t) d\xi_t^c \\ 
     & \qquad \qquad - \sum_{t\in [0,T]} C_{D^0}(t,X_{t-},\xi_{t-},\xi_t) \Big],\\
    &\text{subject to the state dynamics}\\
    &dX_t = b(t,\tilde\mu_t,X_{t-},\xi_{t-}) dt + \sigma(t,\tilde\mu_t,X_{t-},\xi_{t-}) dW_t + \gamma(t,X_{t-},\xi_{t-}) \diamond d\xi_t,\quad X_{0-} = x_{0-}\\
    3.&\text{solve the fixed point problem } \mu^* = \tilde\mu, \text{ where $\mu^*$ denotes the optimal weak control}.
\end{cases}
\label{eq mean-field game singular control general definition}
\end{align}

We emphasise that the above game is still non-standard as the reward function will in general only be u.s.c. To overcome this problem we later introduce MFGs of parametrisations.  The benefit of working with parametrisations is that the reward function is continuous as a function of parametrisations and the set of ``admissible parametrisations'' is compact, at least when restricted to bounded velocity controls.  

To state our existence of equilibrium result, we need to introduce the following additional assumption. \crefcompress{assumption f g continuous,assumption c continuous,assumption f quadratic growth,assumption c linear growth} are standard. The strict growth \cref{assumption g growth} on $g$ with $p > 2$ is needed to ensure the convergence of Nash equilibria in the bounded velocity case to an equilibrium in the unbounded velocity case; a similar approach has been taken in \autocite{fu_mean_2017}.

\begin{assumption}\label{assumptions reward and cost functions}
The reward and cost functions
\[
f:[0,T]\times \c P_2(D^0)\times\R^d\times\R^l\to\R,\quad g:\c P_2(D^0)\times\R^d\times\R^l\to\R\quad\text{and}\quad c:[0,T]\times\R^d\times\R^l\to\R^{1\times l}
\]
satisfy the following conditions.
\begin{assumptionenum}
    \item\label{assumption f g continuous} $f$ and $g$ are continuous.
    \item\label{assumption c continuous} $c$ is locally uniformly continuous.
    \item\label{assumption f quadratic growth} There exists a constant $C_1 > 0$ such that for all $t\in [0,T]$, $m\in\c P_2(\R^d\times\R^l)$ and $(x,\xi)\in\R^d\times\R^l$,
    \[
    |f(t,m,x,\xi)| \leq C_1 (1 + \c W_2^2(m,\delta_0) + |x|^2 + |\xi|^2).
    \]
    \item\label{assumption g growth} There exist $p > 2$ and constants $C_2,C_3 > 0$ such that for all $m\in\c P_2(\R^d\times\R^l)$ and $(x,\xi)\in\R^d\times\R^l$,
    \[
    - C_2(1 + \c W_2^2(m,\delta_0) + |x|^2 + |\xi|^p) \leq g(m,x,\xi) \leq C_3(1 + \c W_2^2(m,\delta_0) + |x|^2 - |\xi|^p).
    \]
    \item\label{assumption c linear growth} There exists a constant $C_4 > 0$ such that for all $t \in [0,T]$ and $(x,\xi)\in \R^d\times\R^l$,
    \[
    |c(t,x,\xi)| \leq C_4(1 + |x| + |\xi|).
    \]
\end{assumptionenum}
\end{assumption}

We are now ready to state the main result of this paper. It  will be obtained as a corollary to a more general result on the existence of equilibria in mean-field games of parametrisations established in \cref{section mean-field game parametrisations}. Specifically, the proof follows by combining \cref{theorem mean-field game parametrisations existence nash equilibrium,theorem mean-field game parametrisation nash equilibrium is singular control nash equilibrium} given below. 

\begin{theorem}\label{theorem mean-field game singular control nash equilibrium}
    Under \cref{assumptions state dynamics,assumptions gamma path independent,assumptions reward and cost functions}, the mean-field game \eqref{eq mean-field game singular control general definition} has a Nash equilibrium.
\end{theorem}



\section{Parametrisations of singular controls}\label{section parametrisations}

The goal of this section is to introduce the concept of parametrisations for the previously defined (weak) singular controls. To this end, we will first recall the notion of a $WM_1$-parametrisation of a càdlàg path, and then extend this idea to singular controls following the approach in \autocite{denkert2023extended}.


\subsection{Parametrisations of càdlàg paths}

Parametrisations are a powerful tool to formalise the idea of interpolating discontinuities of càdlàg functions. Loosely speaking a parametrisation of a càdlàg function is a smoothed modification of that function, running on a different time scale. The precise definition is as follows. 

\begin{definition}[\autocites(Chapter 12){whitt2002stochastic}]\label{definition wm1 parametrisation whitt}
    The \emph{thick graph} $G_y$ of a càdlàg path $y = (x,\xi)\in D^0$ is given by
    \[
    G_y \coloneqq \bigl\{(z,s) \in \R^{d+l}\times [0,T] \bigm\vert z^i \in [y^i(s-)\land y^i(s),y^i(s-)\lor y^i(s)]\text{ for all components }i = 1,\dotso,d+l \bigr\},
    \]
    and equipped with the order relation
    \[
    (z_1,s_1) \leq (z_2,s_2)\text{ if }\begin{cases}s_1 < s_2,\text{ or }\\s_1 = s_2\text{ and }|y^i(s-) - z^i_1| \leq |y^i(s-) - z^i_2|\text{ for all }i=1,\dotso,d+l.\end{cases}
    \]
    A \emph{$WM_1$-parametrisation of $y$} is a continuous non-decreasing (with respect to the above order relation) mapping $(\hat y,\hat r):[0,1] \to G_y $ with
    \[
    \hat y(0) = y(0-),\qquad \hat y(1) = y(T),\qquad \hat r(0) = 0,\qquad \hat r(1) = T.
    \]
    For $y,z\in D^0$ we define
    \[
        d_{WM_1}(y,z) \coloneqq \inf_{\substack{(\hat y,\hat r)\text{ parametrisation of }y\\(\hat z,\hat s)\text{ parametrisation of }z}} \{\norm{\hat y - \hat z}_{\infty} \lor \norm{\hat r - \hat s}_{\infty}\},
    \]
    and say that $y_n\to y$ in the $WM_1$-topology if and only if $d_{WM_1}(y_n,y) \to 0$. 
\end{definition}    

To link parametrised and non-parametrised paths, we denote for any time scale $\bar r :[0,1] \to [0,T]$ its generalised inverse function by
\[
r_{t} \coloneqq \inf\{v \in [0,1] | \bar r_v > t\}.
\]

\begin{definition}\label{definition unparametrisation map s}
    Given a parametrised path $(\bar x,\bar\xi,\bar r)$, we can recover the unparametrised path via the map
    \begin{align}
        \c S:(\bar x,\bar\xi,\bar r) \mapsto (\bar x \circ r,\bar\xi\circ r).
    \end{align}
    We define the domain $\c D(\c S)$ of $\c S$ as the set of all paths $(\bar x,\bar\xi,\bar r)\in C([0,1];\R^d\times\R^l\times [0,T])$ such that
    \begin{enumerate}[label=(\roman*)]
        \item $\bar\xi,\bar r$ are non-decreasing,
        \item $\bar r_0 = 0$ and $\bar r_1 = T$,
        \item for every interval $[a,b]\subseteq [0,1]$ where $\bar r$ is constant, $\bar x$ is monotone in each component on $[a,b]$.
    \end{enumerate}
\end{definition}

The domain $\c D(\c S)$ coincides with the set of all paths $(\bar x,\bar\xi,\bar r)$ that are $WM_1$-parametrisations of some path $(x,\xi)\in D^0([0,T];\R^d\times\R^l)$ with non-decreasing $\xi$, and for any such path it holds that
\[
\c S(\bar x,\bar\xi,\bar r) = (x,\xi).
\]

The following lemma will be key to our subsequent analysis. It shows that the mapping $\c S$ is Lipschitz continuous and that the push-forward mapping 
\[
\nu: \c P_2(\c D(\c S)) \to \c P_2(D^0([0,T];\R^d\times\R^l)), \quad \nu \mapsto  \c S\#\nu
\] 
that maps distributions of parametrised paths into distributions of non-parametrised paths, is continuous. This will allow us to work with continuous reward functionals when analysing MFGs of parametrisations.

\begin{lemma}\label{lemma S continuous}
    The map $\c S$ introduced in \cref{definition unparametrisation map s} is Lipschitz continuous with Lipschitz constant 1 on its domain $\c D(\c S)$. In particular, its push-forward mapping $\nu \mapsto \c S\#\nu$ is continuous. 
\end{lemma}

\begin{proof}
This proof relies on the fact that for each $(\bar x,\bar\xi,\bar r)\in \c D(\c S)$, the path $(\bar x,\bar\xi,\bar r)$ is a $WM_1$-parametrisation of $\c S(\bar x,\bar \xi,\bar r) = (x,\xi)$. Let $(\bar x,\bar\xi,\bar r),(\bar y,\bar\zeta,\bar s)\in \c D(\c S)$ with
\[    
(x,\xi) \coloneqq \c S(\bar x,\bar\xi,\bar r) \quad \mbox{and} \quad (y,\zeta) \coloneqq \c S(\bar y,\bar\zeta,\bar s). 
\]    
Then we obtain from the definition of the $WM_1$-metric that
    \begin{align}
    d_{WM_1}((x,\xi),(y,\zeta))
    &= \inf_{\substack{(\hat X,\hat \xi,\hat r)\text{ $WM_1$-parametrisation of }(x,\xi),\\(\hat y,\hat\zeta,\hat s)\text{ $WM_1$-parametrisation of }(y,\zeta)}} \norm{(\hat x,\hat\xi)-(\hat y,\hat\zeta)}_{\infty} \lor \norm{\hat r - \hat s}_{\infty}\\
    &\leq \norm{(\bar x,\bar\xi) - (\bar y,\bar\zeta)}_{\infty} \lor \norm{\bar r - \bar s}_{\infty} \\
    & \leq \norm{(\bar x,\bar\xi,\bar r) - (\bar y,\bar\zeta,\bar s)}_{\infty}.
    \end{align}
This proves the desired Lipschitz continuity.
\end{proof}


\subsection{Parametrised dynamics, controls and rewards}

Following \autocite{denkert2023extended}, we define parametrisations of (weak) singular controls using parametrised SDEs. This involves transforming the potentially discontinuous state/control process $t \mapsto (X_t,\xi_t)$ into a continuous one by introducing a random time scale $\bar r$ and corresponding continuous processes $(\bar X,\bar \xi)$ on the canonical state space. This mimicks the approach of $WM_1$-parametrisations reviewed in the previous section. 

We expect time-changed processes to satisfy an adapted SDE on the new time scale $\bar r$ and hence start by introducing the notion of weak solutions to time-changed SDEs.

\begin{definition}\label{definition weak solution parametrisation}
    We consider the canonical space
    \[
    \bar\Omega \coloneqq C([0,1];\R^d\times\R^l\times [0,T])\times C([0,T];\R^m),
    \]
    equipped with the Borel $\sigma$-algebra, the canonical processes $(\bar X,\bar\xi,\bar r)$ on the time horizon $[0,1]$ and $\bar W$ on the original time horizon $[0,T]$. Undoing the time-change of $(\bar X,\bar\xi,\bar r)$, we consider the filtration $\bar{\b F}$ on the original time horizon defined by
    \[
    \bar{\c F}_{t} \coloneqq \c F^{\bar W}_{t} \lor \c F^{\bar X,\bar\xi,\bar r}_{r_{t}},\quad t\in [0,T].
    \]
    Given a measure flow $\tilde\mu \in \c P_2(D^0)$, we call a probability measure $\bar\P\in\c P_2(\bar\Omega)$ a weak solution to the time changed SDE (under the time change $\bar r$)
    \begin{align}\label{eq weak solution parametrisation SDE}
        d\bar X_u = b(\bar r_u,\tilde\mu_{\bar r_u},\bar X_u,\bar\xi_u) d\bar r_u + \sigma(\bar r_u,\tilde\mu_{\bar r_u},\bar X_u,\bar\xi_u) d\bar W_{\bar r_u} + \gamma(\bar r_u,\bar X_u,\bar\xi_u) d\bar\xi_u,\quad \bar X_0 = x_{0-},
    \end{align}
    if and only if under $\bar\P$ the following holds:
    \begin{enumerate}[label=(\roman*)]
        \item $\bar r_0 = 0$ and $\bar r_1 = T$, $\bar\P$-a.s.,
        \item $\bar r$ is non-decreasing, $\bar\P$-a.s.,
        \item $\bar W$ is an $\bar{\b F}$-Brownian motion,
        \item $\bar X$ satisfies the dynamics \eqref{eq weak solution parametrisation SDE} driven by the control $\bar\xi$ and the Brownian motion $\bar W$ under the time change $\bar r$.
    \end{enumerate}
\end{definition}

\subsubsection{Controls in parametrisations}

Having introduced weak solutions to parametrised SDEs we are now going to introduce a notion of weak controls in parametrisations. A weak control in parametrisations is a law on the extended path space that corresponds to a weak solution of the parametrised SDE. 

\begin{definition}\label{definition parametrisation}
    Given a measure flow $\tilde\mu\in\c P_2(D^0)$, we define the set of parametrisations $\bar{\c A}(\tilde\mu)$ as the set of all probability measures $\nu\in\c P_2(C([0,1];\R^d\times\R^l\times [0,T])$ such that the following holds: 
    \begin{enumerate}[label=(\roman*)]
        \item $\bar\xi$ is non-decreasing $\nu$-a.s.,
        \item $\bar\xi_0 = 0$, $\nu$-a.s.,
        \item $\nu$ is the marginal law of a weak solution to the time changed SDE \eqref{eq weak solution parametrisation SDE} with the measure flow $\tilde\mu$.
    \end{enumerate}
    With every parametrisation $\nu\in\bar{\c A}(\tilde\mu)$ we associate the unparametrised flow 
    \[
        \mu\coloneqq \c S\#\nu 
\]    
    and call $\nu$ a parametrisation of $\mu$.
    We further say that a parametrisation $\nu$ is Lipschitz continuous with Lipschitz constant $L$ if the processes $\bar\xi$ and $\bar r$ are $\nu$-a.s.\@ Lipschitz continuous with constant $L$.
\end{definition}

In general we cannot expect the unparametrised flow $\mu = \c S\#\nu$ associated with a parametrisation $\nu$ to be an admissible control in the sense of the preceding section. However, the following lemma shows that this is indeed the case whenever \cref{assumptions state dynamics,assumptions gamma path independent}  are satisfied.

\begin{lemma}\label{lemma reducing parametrisations to controls}
Let \cref{assumptions state dynamics,assumptions gamma path independent} hold. For any measure flow $\tilde\mu\in\c P_2(D^0)$, if $\nu\in\bar{\c A}(\tilde\mu)$, then the corresponding unparametrised measure flow is an admissible control, that is 
\[
    \mu\coloneqq \c S\#\nu \in \c A(\tilde\mu).
\]
\end{lemma}

\begin{proof}
    Given $\nu\in\bar{\c A}(\tilde\mu)$, let $\bar\P\in\c P_2(\bar\Omega)$ be a corresponding weak solution to the time changed SDE \eqref{eq weak solution parametrisation SDE}. Then the process $(X,\xi) \coloneqq \c S(\bar X,\bar\xi,\bar r) = (\bar X\circ r,\bar\xi\circ r)$ satisfies for all $t\in [0,T]$ that
    \begin{align}
        X_t &= x_{0-} + \int_0^{r_t} b(\bar r_u,\tilde\mu_{\bar r_u},\bar X_u,\bar\xi_u) d\bar r_u + \int_0^{r_t}\sigma(\bar r_u,\tilde\mu_{\bar r_u},\bar X_u,\bar\xi_u) d\bar W_{\bar r_u} + \int_0^{r_t}\gamma(\bar r_u,\bar X_u,\bar\xi_u) d\bar\xi_u\\
        &= x_{0-} + \int_0^t b(s,\tilde\mu_s,X_s,\xi_s)ds + \int_0^t \sigma(s,\tilde\mu_s,X_s,\xi_s)d\bar W_s+ \int_0^t \gamma(s,X_s,\xi_s) d\xi^c_s\\
        &\quad + \sum_{s\in [0,t], \xi_{s-}\not=\xi_s} \int_{r_{s-}}^{r_s} \gamma(\bar r_u,\bar X_u,\bar\xi_u) d\bar\xi_u.
    \end{align}
    
    On each such jump interval $[r_{s-},r_s]$, the parametrised process $\bar\xi$ is a continuous, monotone interpolation between $\xi_{s-}$ and $\xi_s$ and the state process $\bar X$ satisfies the Marcus-type jump dynamics \eqref{eq marcus integration jump} with $\zeta \coloneqq \bar\xi$ on $[r_{s-},r_s]$. As a result, and recalling that $(X,\xi)$ is càdlàg, it follows from \cref{assumptions gamma path independent} that
    \[
    dX_t = b(t,\tilde\mu_t,X_{t-},\xi_{t-})dt + \sigma(t,\tilde\mu_t,X_{t-},\xi_{t-})d\bar W_t + \gamma(t,X_{t-},\xi_{t-}) \diamond d\xi_t,\quad X_{0-} = x_{0-},
    \]
    and thus $\mu = \bar\P_{(X,\xi)} \in \c A(\tilde\mu)$.
\end{proof}

The preceding result showed that under \cref{assumptions state dynamics,assumptions gamma path independent}  any parametrisation yields an admissible control. The next lemma, whose proof is given in \cref{appendix proof lemma wm1 convergence parametrisation convergence}, shows that every limit of continuous controls admits a parametrisation that can be obtained as the limit of parametrisations of the approximating continuous controls. Together with \cref{lemma reducing parametrisations to controls} this shows that under \cref{assumptions state dynamics,assumptions gamma path independent} every limit of continuous controls is an admissible (singular) control i.e.~that $$\overline{\c C(\tilde\mu)}\subseteq \c A(\tilde\mu).$$ 

In preparation for the fixed point argument required to solve our MFGs we state the lemma for general sequences of ``exogenous'' measure flows $\tilde\mu^n\to\tilde\mu$.

\begin{lemma}\label{lemma wm1 convergence parametrisation convergence}
    Let \cref{assumptions state dynamics} hold and let $\tilde\mu^n\to\tilde\mu$ in $\c P_2(D^0)$ be a convergent sequence of measure flows. Let $(\mu^n)_n\in \c P_2(C)$ be a sequence of \emph{continuous} controls with $\mu^n\in\c C(\tilde\mu^n)\subseteq\c A(\tilde\mu^n)$ and 
    \[
        \mu^n\to\mu \quad \mbox{in} \quad \c P_2(D^0). 
\]        
Then there exists a parametrisation $\nu \in\bar{\c A}(\tilde\mu)$ of $\mu$ and a sequence of parametrisations $(\nu^n)_n$ such that each $\nu^n \in\bar{\c A}(\tilde\mu^n)$ is a parametrisation of $\mu^n$ and, along a subsequence, 
\[
    \nu^n\to\nu \quad \mbox{in} \quad \c P_2(C).
\]
\end{lemma}

The next result allows us to approximate admissible controls by Lipschitz continuous ones. The proof is given in \cref{appendix proof lemma parametrisation construct lipschitz approximations}. Specifically, it shows that any measure flow $\mu\in \c P_2(D^0)$ that admits a parametrisation $\nu$ can be approximated by a sequence $(\mu^n)_n$ of \emph{Lipschitz continuous} controls with corresponding parametrisations $(\nu^n)_n$ in such a way that the sequence $(\mu^n,\nu^n)_n$ converges to $(\mu,\nu)$.
This shows that 
\[
    \c A(\tilde\mu) \subseteq \overline{\c C(\tilde\mu)}.
\]

\begin{lemma}\label{lemma parametrisation construct lipschitz approximations}
    Let \cref{assumptions state dynamics} hold and let $\tilde\mu\in\c P_2(D^0)$ be a given measure flow. Let $\nu\in\bar{\c A}(\tilde\mu)$ be a parametrisation of $\mu\in \c P_2(D^0)$. Then for every sequence of approximating measure flows 
    \[
    \tilde\mu^n\to \tilde\mu \quad \mbox{in} \quad \c P_2(D^0), 
    \]    
    there exists an increasing sequence $(k_n)_n \subseteq \N$ with $k_n\to\infty$ and a sequence of $k_n$-Lipschitz admissible controls $\mu^n\in \c A(\tilde\mu^{k_n})$ with $k_n$-Lipschitz parametrisations $\nu^n\in\bar{\c A}(\tilde\mu^{k_n})$ such that
    \[
    (\mu^n,\nu^n) \to (\mu,\nu)\quad \text{in} \quad \c P_2(D^0)\times \c P_2(C).
    \]    
    Furthermore, if
    \[
        (\mu_\xi,\nu_{\bar\xi}) \in \c P_q(D^0([0,T];\R^l))\times\c P_q(C([0,1];\R^l)) 
    \]
    for some $q > 2$, then we can choose the approximating sequence $(\mu^n,\nu^n)_n$ to satisfy
    \[
    (\mu^n_\xi,\nu^n_{\bar\xi}) \to (\mu_\xi,\nu_{\bar\xi})\quad\text{in} \quad \c P_q(D^0([0,T];\R^l))\times\c P_q(C([0,1];\R^l)).
    \]
\end{lemma}


\subsubsection{Rewards in parametrisations}

Having introduced the concept of controls in parametrisations we now define the reward of a parametrisation as 
\[
J(\tilde\mu,\nu) = \E^\nu\Big[g(\tilde\mu_T,\bar X_1,\bar\xi_1) + \int_0^1 f(\bar r_u,\tilde\mu_{\bar r_u},\bar X_u,\bar\xi_u) d\bar r_u - \int_0^1 c(\bar r_u,\bar X_u,\bar\xi_u) d\bar\xi_u \Big].
\]

In view of \cref{assumption g growth} there exists $p>2$ such that 
\[
    J(\tilde\mu,\nu) = -\infty \quad \mbox{if} \quad \nu_{\bar\xi}\not\in\c P_p(C([0,1];\R^l)).
\]    
In particular, parametrisation $\nu$ for which $\nu_{\bar\xi}\not\in\c P_p(C([0,1];\R^l))$ are not relevant for our MFG and can be disregarded when analysing continuity properties of the reward function. 

It follows from the preceding lemma that any ``relevant'' parametrisation can be approximated by parametrisations associated with Lipschitz continuous admissible controls such that 
\[
    \nu^n_{\bar\xi}\to\nu_{\bar\xi} \quad \text{in} \quad \c P_p(C([0,1];\R^l)). 
\]    
The next lemma shows that the reward function in parametrisations is continuous (in a stronger topology) on the subset of ``relevant'' parametrisations of the graph
\[
\Gamma_{\bar{\c A}}\coloneqq \{(\tilde\mu,\nu) \in \c P_2(D^0)\times\c P_2(C) \mid \nu \in \bar{\c A}(\tilde\mu) \}
\]
of the set-valued mapping 
\[
\bar{\c A}:\c P_2(D^0([0,T];\R^d\times\R^l)) \rightrightarrows \c P_2(C([0,1];\R^d\times\R^l\times [0,T])).
\]

\begin{lemma}\label{lemma J parametrisations continuous}
If \cref{assumptions state dynamics,assumptions reward and cost functions} are satisfied, then the following holds.
\begin{enumerate}[wide,label=(\roman*)]
    \item 
    The reward functional $J:\Gamma_{\bar{\c A}}\to \R$ is upper semi-continuous.
    \item
    If $(\tilde\mu^n,\nu^n) \to (\tilde\mu,\nu)$ in $\Gamma_{\bar{\c A}}$ such that $\nu^n_{\bar\xi} \to \nu_{\bar\xi}$ in $\c P_p(C([0,1];\R^l))$, then $J(\tilde\mu^n,\nu^n) \to J(\tilde\mu,\nu)$.
\end{enumerate}
\end{lemma}

\begin{proof}
Let $(\tilde\mu^n,\nu^n)_n\cup (\tilde\mu,\nu) \subseteq \Gamma_{\bar{\c A}}$ with 
\begin{equation}\label{mun-mu}
    (\tilde\mu^n,\nu^n)\to (\tilde\mu,\nu) \quad \mbox{in} \quad \c P_2(D^0)\times \c P_2(C). 
\end{equation}
Since the set of continuous functions $C([0,1];\R^d\times\R^l\times [0,T])$ is separable, we can use Skorokhod's representation theorem, see \autocites(Theorem 6.7){billingsley2013convergence}, to obtain processes $(\bar X^n,\bar \xi^n,\bar r^n)_n$ and $(\bar X,\bar\xi,\bar r)$ on a joint probability space $(\tilde\Omega,\tilde {\c F},\tilde\P)$ such that $\tilde\P_{(\bar X^n,\bar\xi^n,\bar r^n)} = \nu^n$, $\tilde\P_{(\bar X,\bar \xi,\bar r)} = \nu$ and 
\[
(\bar X^n,\bar\xi^n,\bar r^n) \to (\bar X,\bar\xi,\bar r)\text{ in }C([0,1];\R^d\times\R^l\times [0,T])\quad\tilde\P\text{-a.s.\@ and in }L^2,
\]
which in particular implies that the sequence
$
\big(\sup_{u\in [0,1]} |\bar X^n_u|^2 + \sup_{u\in [0,1]} |\bar \xi^n_u|^2\big)_n
$
is uniformly integrable.

We can estimate the differences in the rewards by
\begin{align}
&|J(\tilde\mu^n,\nu^n) - J(\tilde\mu,\nu)|\\
&\leq \E^{\tilde\P}\Big[\Big|\int_0^1 f(\bar r^n_u,\tilde\mu^n_{\bar r^n_u},\bar X^n_u,\bar\xi^n_u) d\bar r^n_u - \int_0^1 f(\bar r_u,\tilde\mu_{\bar r_u},\bar X_u,\bar\xi_u) d\bar r_u\Big|\\
&\qquad + |g(\tilde\mu^n_T,\bar X^n_1,\bar\xi^n_1) - g(\tilde\mu_T,\bar X_1,\bar\xi_1)| + \Big|\int_0^1 c(\bar r^n_u,\bar X^n_u,\bar\xi^n_u) d\bar\xi^n_u - \int_0^1 c(\bar r_u,\bar X_u,\bar\xi_u)d\bar\xi_u\Big| \Big].
\end{align}
Since by \cref{assumption f g continuous} the terminal payoff function $g$ is continuous,
\[
g(\tilde\mu_T^n,\bar X^n_1,\bar\xi^n_1)\to g(\tilde\mu_T,\bar X_1,\bar\xi_1),\qquad\tilde\P\text{-a.s.}
\]
Together with \cref{assumption g growth} and Fatou's lemma this implies that
\[
\limsup_{n\to\infty} \E^{\tilde\P}[g(\tilde\mu^n_T,\bar X^n_1,\bar\xi^n_1)] \leq \E^{\tilde\P}[g(\tilde\mu_T,\bar X_1,\bar\xi_1)].
\]
If the stronger convergence $\nu^n_{\bar\xi} \to \nu_{\bar\xi}$ in $\c P_p(C([0,1]; \R^l ))$ holds, then by applying the following consequence of \cref{assumption g growth},
\[
    |g(\tilde\mu^n,\bar X^n_1,\bar\xi^n_1)| \leq C ( 1 + \c W_2^2(\tilde\mu^n,\delta_0) + |\bar X^n_1|^2 + |\bar\xi^n_1|^p),
\]
we can conclude that $(g(\tilde\mu^n,\bar X^n_1,\bar\xi^n_1))_n$ is uniform integrable.
Thus it follows from Vitali's convergence theorem that 
\[
    \E^{\tilde\P}[g(\tilde\mu^n_T,\bar X^n_1,\bar\xi^n_1)] \to \E^{\tilde\P}[g(\tilde\mu_T,\bar X_1,\bar\xi_1)].
\]
Along with the arguments given below this shows that the reward function is u.s.c.\@ and continuous if the stronger convergence $\nu^n_{\bar\xi} \to \nu_{\bar\xi}$ in $\c P_p(C([0,1]; \R^l ))$ holds.

We now turn to the singular cost term. It satisfies
\begin{align}
&\Big|\int_0^1 c(\bar r^n_u,\bar X^n_u,\bar\xi^n_u) d\bar \xi^n_u - \int_0^1 c(\bar r_u,\bar X_u,\bar\xi_u) d\bar \xi_u\Big|\\
&\leq \int_0^1 |c(\bar r^n_u,\bar X^n_u,\bar\xi^n_u) - c(\bar r_u,\bar X_u,\bar\xi_u)| d\bar \xi^n_u
+ \Big|\int_0^1 c(\bar r_u,\bar X_u,\bar\xi_u) d\bar\xi^n_u - \int_0^1 c(\bar r_u,\bar X_u,\bar\xi_u) d\bar \xi_u\Big|,\quad \tilde\P\text{-a.s.}
\end{align}
The first term can be bounded by
\[
\int_0^1 |c(\bar r^n_u,\bar X^n_u,\bar\xi^n_u) - c(\bar r_u,\bar X_u,\bar\xi_u)| d\bar\xi^n_u \leq \sup_{u\in [0,1]} |c(\bar r^n_u,\bar X^n_u,\bar\xi^n_u) - c(\bar r_u,\bar X_u,\bar\xi_u)||\bar\xi^n_1|,\quad\tilde\P\text{-a.s.}
\]
The local uniform continuity of $c$ by \cref{assumption c continuous} implies $\tilde\P$-a.s.\@ convergence to zero. Furthermore, the linear growth in $(x,\xi)$ from \cref{assumption c linear growth} implies that 
\[
    \sup_{u\in [0,1]} |c(\bar r^n_u,\bar X^n_u,\bar\xi^n_u) - c(\bar r_u,\bar X_u,\bar\xi_u)||\bar\xi^n_1| \leq 
    C\Big( 1 + \sup_{u\in [0,1]} |\bar X^n_u|^2 + \sup_{u\in [0,1]} |\bar \xi^n_u|^2 \Big)
    \quad\tilde\P\text{-a.s.},
\]
from which we deduce the uniform integrability of the singular cost term. Its $L^1$-convergence follows again from Vitali's convergence theorem. The second term vanishes $\tilde\P$-a.s.\@ due to the Portmanteau theorem. Using Vitali's convergence theorem again, this term also vanishes in $L^1$.

To deal with the running reward, we deduce from \cref{lemma S continuous} that 
\[
(X^n,\xi^n)\coloneqq \c S(\bar X^n,\bar\xi^n,\bar r^n)\to \c S(\bar X,\bar\xi,\bar r) \eqqcolon (X,\xi)\text{ in }D^0([0,T];\R^d\times\R^l)\quad\tilde\P\text{-a.s.\@ and in }L^2.
\]
Moreover, since
\[
\sup_{t\in [0,T]} |X^n_t|^2 + \sup_{t\in [0,T]} |\xi^n_t|^2 \leq \sup_{u\in [0,1]} |\bar X^n_u|^2 + \sup_{u\in [0,1]} |\bar \xi^n_u|^2,
\]
we see that 
$
\big(\sup_{t\in [0,T]} |X^n_t|^2 + \sup_{t\in [0,T]} |\xi^n_t|^2\big)_n
$
is also uniformly integrable. 
By rewriting running reward part in terms of $(X^n,\xi^n)$ and $(X,\xi)$, we see that
\begin{align}
&\E^{\tilde\P}\Big[\Big|\int_0^1 f(\bar r^n_u,\tilde\mu^n_{\bar r^n_u},\bar X^n_u,\bar\xi^n_u) d\bar r^n_u - \int_0^1 f(\bar r_u,\tilde\mu_{\bar r_u},\bar X_u,\bar\xi_u) d\bar r_u\Big|\Big]\\
&= \E^{\tilde\P}\Big[\Big|\int_0^T f(t,\tilde\mu^n_t,X^n_t,\xi^n_t) dt - \int_0^T f(t,\tilde\mu_t,X_t,\xi_t) dt\Big|\Big] \\
& \leq \E^{\tilde\P}\Big[\int_0^T |f(t,\tilde\mu^n_t,X^n_t,\xi^n_t)-f(t,\tilde\mu_t,X_t,\xi_t)| dt\Big].
\end{align}
Finally the quadratic growth and continuity of $f$ from \cref{assumption f quadratic growth,assumption f g continuous} imply by Vitali's convergence theorem the desired convergence.
\end{proof}

We are now ready to establish our representation result \eqref{J1} for the reward functional.

\begin{theorem}\label{lemma J control J sup parametrisations}
    Let \cref{assumptions state dynamics,assumptions gamma path independent,assumptions reward and cost functions} hold.
    We have for every $\tilde\mu\in\c P_2(D^0)$ and $\mu\in\c A(\tilde\mu)$,
    \begin{align}
    J(\tilde\mu,\mu) &= \max_{\nu\in\bar{\c A}(\tilde\mu)\text{ parametrisation of }\mu} J(\tilde\mu,\nu)\\
    &= \E^\mu \Big[g(\tilde\mu_T,X_T,\xi_T) + \int_0^T f(t,\tilde\mu_t,X_t,\xi_t) dt
     - \int_0^T c(t,X_t,\xi_t) d\xi_t^c - \sum_{t\in [0,T]} C_{D^0}(t,X_{t-},\xi_{t-},\xi_t)\Big],
     \label{eq lemma J control J sup parametrisations}
    \end{align}
    where $C_{D^0}$ is defined as in \cref{definition C_D0}.
\end{theorem}
\begin{proof}
    Since $J(\tilde\mu,\mu) = -\infty$ for all $\mu_\xi\not\in\c P_p(D^0([0,T];\R^l))$ by \cref{assumptions reward and cost functions}, we only need to consider the case $$\mu_\xi\in\c P_p(D^0([0,T];\R^l)).$$ In view of \cref{lemma parametrisation construct lipschitz approximations,lemma wm1 convergence parametrisation convergence}, by choosing $\tilde\mu^n = \tilde\mu$ for all $n\in\N$, and the continuity of the reward function in parametrisations established in \cref{lemma J parametrisations continuous}, we obtain for all $\tilde\mu\in \c P_2(D^0)$ and $\mu\in\c A(\tilde\mu)$ that
    \begin{align}
    J(\tilde\mu,\mu) = \sup_{\nu\in\bar{\c A}(\tilde\mu)\text{ parametrisation of }\mu} J(\tilde\mu,\nu).
    \label{eq lemma J control J sup parametrisations first part}
    \end{align}
  Using \cref{assumptions gamma path independent} and following the same arguments given in the proof of \autocite[Theorem 3.4]{denkert2023extended} shows that the second equality in \eqref{eq lemma J control J sup parametrisations} for the second layer shows that the second equality in \eqref{eq lemma J control J sup parametrisations} also holds and furthermore that the supremum in \eqref{eq lemma J control J sup parametrisations first part} is attained.
  \end{proof}


\section{MFGs of parametrisations}\label{section mean-field game parametrisations}

In this section we introduce MFGs of parametrisations, that is, MFGs where the set of admissible controls is given by parametrised measure-flows. We have seen that parametrisations are a natural way to ``smooth'' singular controls by ``incorporating additional information on how jumps in the control variable are executed'' and that the reward functional is continuous on the set of ``relevant'' parametrisations. 

We prove that under \cref{assumptions state dynamics,assumptions reward and cost functions} any MFG of parametrisations admits a Nash equilibrium for general impact and cost functions $\gamma$ and $c$. To show that any equilibrium in parametrisations induces an equilibrium in the underlying MFGs with singular controls, \cref{assumptions gamma path independent} is required. 

One additional subtlety arises when working with parametrisations. Parametrisations come with their own time scale that might be different for different states of the world. As a result, we first need to reverse the time change using the mapping $\c S$ introduced in \cref{section parametrisations}. This results in an additional step in the formulation of MFG of parametrisations. Specifically, the MFG of parametrisations is defined as follows: 
\begin{align}\label{eq mean-field game parametrisations general definition}
\begin{cases}
    1.&\text{fix a parametrised measure flow }\tilde\nu \in \c P_2(C([0,1];\R^d\times\R^l\times [0,T])),\\
    2.&\text{recover the unparametrised measure flow }\tilde\mu \coloneqq \c S\#\tilde\nu \in \c P_2(D^0([0,T];\R^d\times\R^l)),\\
    3.&\text{solve the stochastic optimisation problem}\\
    &\sup_{\nu \in \bar{\c A}(\tilde\mu)} \E^\nu\Big[g(\tilde\mu_T,\bar X_1,\bar\xi_1) + \int_0^1 f(\bar r_u,\tilde\mu_{\bar r_u},\bar X_u,\bar\xi_u) d\bar r_u - \int_0^1 c(\bar r_u,\bar X_u,\bar\xi_u) d\bar\xi_u \Big],\\
    &\text{subject to the state dynamics}\\
    &d\bar X_u = b(\bar r_u,\tilde\mu_{\bar r_u},\bar X_u,\bar\xi_u) d\bar r_u + \sigma(\bar r_u,\tilde\mu_{\bar r_u},\bar X_u,\bar\xi_u) d\bar W_{\bar r_u} + \gamma(\bar r_u,\bar X_u,\bar\xi_u) d\bar\xi_u,\quad \bar X_0 = x_{0-}\\
    4.&\text{solve the fixed point problem }\nu^* = \c L(\bar X,\bar\xi,\bar r) = \tilde\nu.
\end{cases}
\end{align}

We are now ready to state the main result of this paper. The proof follows from \cref{lemma mean-field game nash equilibria sequence uniform estimate,lemma mean-field game parametrisations accumulation point,theorem mean-field game parametrisations accumulation nash equilibrium} given below. 


\begin{theorem}\label{theorem mean-field game parametrisations existence nash equilibrium}
    Under \cref{assumptions state dynamics,assumptions reward and cost functions} the mean-field game of parametrisations \eqref{eq mean-field game parametrisations general definition} admits a Nash equilibrium.
\end{theorem}

Combined with the following \cref{theorem mean-field game parametrisation nash equilibrium is singular control nash equilibrium}, the above result yields the existence of a Nash equilibrium the original mean-field game \eqref{eq mean-field game singular control general definition} of singular controls. The proof of \cref{theorem mean-field game parametrisation nash equilibrium is singular control nash equilibrium} follows from \cref{lemma reducing parametrisations to controls,lemma J control J sup parametrisations}.

\begin{theorem}\label{theorem mean-field game parametrisation nash equilibrium is singular control nash equilibrium}
    If \cref{assumptions state dynamics,assumptions gamma path independent,assumptions reward and cost functions} are satisfied, then the following holds.
    \begin{enumerate}[wide, label=(\roman*)]
        \item 
        If $\mu$ is a Nash equilibrium of the mean-field game \eqref{eq mean-field game singular control general definition}, then there exists a parametrisation $\nu\in\bar{\c A}(\mu)$ such that $J(\mu,\mu) = J(\mu,\nu)$ and $\nu$ is a Nash equilibrium of the mean-field game \eqref{eq mean-field game parametrisations general definition} of parametrisations.
        \item
        Conversely, if $\nu$ is a Nash equilibrium of the mean-field game of parametrisations \eqref{eq mean-field game parametrisations general definition}, then the measure flow $\mu \coloneqq \c S\#\nu$ is a Nash equilibrium of the mean-field game of singular controls \eqref{eq mean-field game singular control general definition} and $J(\mu,\mu) = J(\mu,\nu)$.
    \end{enumerate}
\end{theorem}



\subsection{The bounded velocity case}

As a first step towards the proof of \cref{theorem mean-field game parametrisations existence nash equilibrium} we restrict ourselves to bounded velocity controls 
\[
    \xi_t = \int_0^t u_s ds,\qquad t\in [0,T],
\]
with common Lipschitz constant $K > 0$. That is,  
\[
    u_s \geq 0, \quad |u_s| \leq K \quad \mbox{for all} \quad s\in [0,T].
\]

The advantage of working with bounded velocity controls is that the set of admissible controls is compact. Since every continuous control admits exactly one parametrisation up to a reparametrisation of time, MFGs of parametrisations can -- and should -- be formulated within the standard setting of MFG theory when continuous controls are considered. Thus, we consider the following standard MFG with regular control $u$: 
%
\begin{align}
    \label{eq mean-field game regular controls}
    \begin{cases}
    1.&\text{fix a measure flow }\tilde\mu\in\c P_2(C([0,T];\R^d\times\R^l)),\\
    2.&\text{solve the stochastic optimisation problem}\\
    &\sup_\mu \E^\mu\Big[g(\tilde\mu_T,X_T,\xi_T) + \int_0^T \bigl( f(t,\tilde\mu_t,X_t,\xi_t) - c(t,X_t,\xi_t) u_t\bigr) dt \Big],\\
    &\parbox[t]{.8\textwidth}{over all $\mu\in\c C(\tilde\mu)$ such that $u$ exists, $u_t\geq 0$ and $|u_t|\leq K$ for all $t\in [0,T]$, $\mu$-a.s., subject to the state dynamics}\\
    &dX_t = \bigl(b(t,\tilde\mu_t,X_t,\xi_t) + \gamma(t,X_t,\xi_t) u_t\bigr) dt + \sigma(t,\tilde\mu_t,X_t,\xi_t) dW_t,\quad X_{0-} = x_{0-},\\
    &d\xi_t = u_t dt,\quad \xi_{0-} = 0,\\
    3.&\text{solve the fixed point problem }\mu^* = \c L(X,\xi) = \tilde\mu.
    \end{cases}
\end{align}
We call the above MFG a $K$-bounded velocity MFG. Such games have been extensively studied in the recent literature. In particular, we have the following result. 


\begin{theorem}[{\autocite[Corollary 3.8]{lacker_mean_2015}}]
    Under \cref{assumptions state dynamics,assumptions reward and cost functions}, the $K$-bounded velocity MFG \eqref{eq mean-field game regular controls} admits a Nash equilibrium. 
\end{theorem}


\subsection{The unbounded case}

Having established the existence of an equilibrium in the bounded velocity case, we are now going to prove that a Nash equilibrium in parametrisations in the unconstrained case can be obtained in terms of weak limits of sequences of bounded velocity equilibria $(\mu^n)_n$ with Lipschitz constants $K=n$. 

In a first step we prove in \cref{lemma mean-field game nash equilibria sequence uniform estimate,lemma mean-field game parametrisations accumulation point} that the sequence $(\mu^n)_n$ is bounded w.r.t.~the 2-Wasserstein distance and relatively compact in $\c P_2(D^0)$. It hence admits a weak accumulation point $\mu$. 

In a second step we use that by \cref{lemma wm1 convergence parametrisation convergence} there exist admissible parametrisations $\nu^n$ of $\mu^n$ that converge to an admissible parametrisation $\nu$ of $\mu$ along a suitable subsequence. Using the continuity of the reward function in parametrisations we can then prove that $\nu$ is an equilibrium in the MFG of parametrisations and hence with \cref{theorem mean-field game parametrisation nash equilibrium is singular control nash equilibrium} that $\mu$ is an equilibrium in the underlying MFG with singular controls.  


\begin{lemma}\label{lemma mean-field game nash equilibria sequence uniform estimate}
    Let \cref{assumptions state dynamics,assumptions reward and cost functions} hold and let $(\mu^n)_n$ be a sequence of bounded velocity Nash equilibria with respective Lipschitz constants $K=n$ of the $n$-bounded velocity mean-field games \eqref{eq mean-field game regular controls}. Then, 
    \begin{align}
        \sup_n \c W_p^p(\mu^n,\delta_0) = \sup_n \E^{\mu^n}\Big[\sup_{t\in [0,T]} |X_t|^p + |\xi_T|^p\Big] < \infty.
    \end{align}
\end{lemma}
\begin{proof}
    Let us fix an $n\in\N$. Since $\mu^n\in\c C(\mu^n)$ is a continuous, admissible control, there exists a probability measure $\P\in\c P_2(\tilde\Omega)$ with $\mu^n = \P_{(X,\xi)}$ under which the state dynamics satisfies 
    \begin{align}
        dX_t &= b(t,\mu^n_t,X_t,\xi_t) dt + \sigma(t,\mu^n_t,X_t,\xi_t) d W_t + \gamma(t,X_t,\xi_t)  d\xi_t,\quad t\in [0,T],\quad X_{0-} = x_{0-}.
    \end{align}
    
    Let $X^{n,0}$ be the solution to the state dynamics for the constant control $\xi \equiv 0$, that is,
    \begin{align}
        dX^{n,0}_t &= b(t,\mu^n_t,X^{n,0}_t,0) dt + \sigma(t,\mu^n_t,X^{n,0}_t,0) d W_t,\quad X^{n,0}_0 = x_{0-},
    \end{align}
    and let $\mu^{n,0} \coloneqq \P_{(X^{n,0},0)}$ be the corresponding weak control. By construction $\mu^{n,0}\in \c C(\mu^n)$ and so 
    \[
    J(\mu^n,\mu^{n,0}) \leq J(\mu^n,\mu^n).
    \]
    In view of \cref{assumptions reward and cost functions}
    \begin{align}
        J(\mu^n,\mu^{n,0})
        &= \E^{\mu^{n,0}}\Big[\int_0^T f(t,\mu^n_t,X_t,0) dt + g(\mu^n_T,X_T,0) \Big]\\
        &\geq - C_f\int_0^T \Big(1 + \c W_2^2(\mu^n_t,\delta_0) + \E^{\mu^{n,0}}[|X_t|^2]\Big) dt - C_{g,1}\Big(1 + \c W_2^2(\mu^n_T,\delta_0) + \E^{\mu^{n,0}}[|X_T|^2]\Big),
    \end{align}
    and
    \begin{align}
        J(\mu^n, \mu^n)
        &\leq C_f\int_0^T \Big(1 + \c W_2^2(\mu^n_t,\delta_0) + \E^{\mu^n}[|X_t|^2 + |\xi_t|^2]\Big) dt + C_{g,1}\Big(1 + \c W_2^2(\mu^n_T,\delta_0) + \E^{\mu^n}[|X_T|^2]\Big)\\
        &\quad- C_{g,2} \E^{\mu^n}[|\xi_T|^p]
        + C_c \E^{\mu^n}\Big[\int_0^T (1 + |X_t| + |\xi_t|) d|\xi_t|\Big],
    \end{align}
    where the constants $C_f,C_{g,1},C_{g,2},C_c$ depend only on $f,g$ respectively $c$.
    Putting these results together, we obtain that
    \begin{align}
         \E^{\mu^n}[|\xi_T|^p]
         &\leq C_{f,g}\int_0^T (1 + \E^{\mu^n}[|X_t|^2 + |\xi_t|^2] + \E^{\mu^{n,0}}[|X_t|^2]) dt\\
         &\quad+ C_{f,g} (1 + \E^{\mu^n}[|X_T|^2 + |\xi_T|^2] + \E^{\mu^{n,0}}[|X_T|^2])\\
         &\quad+ C_{c,g} \Big(1 + \E^{\mu^n}\Big[\sup_{t\in [0,T]} |X_t|^2 + |\xi_T| + |\xi_T|^2\Big] \Big)\\
         &\leq C_{c,f,g} \Bigl(1 + \E^{\mu^n}\Big[\sup_{t\in [0,T]} |X_t|^2 + |\xi_T| + |\xi_T|^2\Big] + \sup_{t\in [0,T]} \E^{\mu^{n,0}}[|X_t|^2]\Bigr).
    \end{align}
    To further bound the right-hand side, we apply a standard Gronwall-type argument to get that
    \[
    \E^{\mu^n}\Bigl[\sup_{t\in [0,T]} |X_t|^{p'}\Bigr]
    \leq C ( 1 + \E^{\mu^n}[|\xi_T|^{p'}] ),
    \qquad\text{for all }p'\in [2,p],
    \]
    and
    \[
    \E^{\mu^{n,0}}\Bigl[\sup_{t\in [0,T]} |X_t|^2\Bigr]
    \leq C \Bigl( 1 + \sup_{t\in [0,T]} \c W_2^2(\mu^n_t,\delta_0) \Bigr)
    \leq C ( 1 + \E^{\mu^n}[|\xi_T|^2] ),
    \]
    where the $C$ is independent of $n$. Hence, 
    \[
        \E^{\mu^n}\Big[\sup_{t\in [0,T]} |X_t|^p + |\xi_T|^p\Big]
        \leq C (1 + \E^{\mu^n}[|\xi_T| + |\xi_T|^2]).
    \]
    Since $p>2$, this implies that 
    \[
    \sup_n \E^{\mu^n}\Big[\sup_{t\in [0,T]} |X_t|^p + |\xi_T|^p\Big] < \infty.
    \]
\end{proof}

The next lemma shows that the sequence of bounded velocity equilibria is relatively compact. 

\begin{lemma}\label{lemma mean-field game parametrisations accumulation point}
    Let \cref{assumptions state dynamics} hold and let $q > 2$. Let $(\tilde\mu^n)_n\subseteq \c P_q(D^0)$ be a sequence of measure flows and $(\mu^n)_n$ be a sequence of corresponding continuous controls with $\mu^n\in\c C(\tilde\mu^n)$ satisfying
    \begin{align}\label{eq lemma 3.4 uniform q estimate}
        \sup_n \big[\c W_q^q(\tilde\mu^n,\delta_0) + \c W_q^q(\mu^n,\delta_0)\big] < \infty.
    \end{align}
    Then $(\mu^n)_n\subseteq \c P_2(D^0)$ is relatively compact.
\end{lemma}

\begin{proof}
    Since the control process $\xi$ is non-decreasing under $\mu^n$, for all $n\in\N$, the relative compactness of the sequence $(\mu^n_\xi)_n \subseteq \c P_2(D^0([0,T];\R^l))$ follows directly from \eqref{eq lemma 3.4 uniform q estimate}.
    
    Let us now turn to the state process $X$. For any sequence of continuous weak controls $(\mu^n)_n$ with $\mu^n\in \c C(\tilde\mu^n)\subseteq \c A(\tilde\mu^n)$ there exists a sequence of corresponding weak solutions $(\P^n)_n\subseteq \c P_2(\tilde\Omega)$ to our Marcus-type SDE \eqref{eq singular control state dynamics}. Now let us decompose the process $X$ into
    \[
    \Gamma_t \coloneqq \int_0^t \gamma(s,X_{s},\xi_{s}) \diamond d\xi_{s},\qquad t\in [0,T],
    \]
    and
    \[
    L_t \coloneqq X_t - x_{0-} - \Gamma_t,\qquad t \in [0,T].
    \]
    Since the state process $X$ satisfies the dynamics \eqref{eq singular control state dynamics} $\P^n$-a.s.\@ with the given measure flow $\tilde\mu^n$, we have that 
    \[
    L_t = \int_0^t b(s,\tilde\mu^n_{s},X_{s},\xi_{s}) ds + \int_0^t \sigma(s,\tilde\mu^n_{s},X_{s},\xi_{s}) dW_{s},\qquad t\in [0,T],\qquad \P^n\text{-a.s}.
    \]
   In view of
   \eqref{eq lemma 3.4 uniform q estimate}, 
    \[
    \sup_n \E^{\mu^n}\Bigl[\sup_{t\in [0,T]} (|b(t,\tilde\mu^n_{t},X_{t},\xi_{t})|^q + |\sigma(t,\tilde\mu^n_{t},X_{t},\xi_{t})|^q)\Bigr] < \infty,
    \]
    and so it follows from \autocite{zheng1995tightness} that the sequence $(\P^n_L)_n$ is relatively compact in $\c P_2(C([0,T];\R^d))$. 

    To show that the sequence $(\P^n_{\Gamma})_n$ is relatively compact in $\c P_2(D^0([0,T];\R^d))$ we construct explicit $WM_1$-parametrisations as follows. We start by introducing the strictly monotone function
    \[
    r_{t} \coloneqq \frac{t + \arctan(\Var(\xi,[0,t]))}{T + \frac\pi 2},\qquad t\in [0,T],
    \]
    where $\Var$ denotes the total Variation. We denote its generalised inverse we denote by
    \[
    \bar r_v \coloneqq \inf\{t\in [0,T] | r_{t} > v\} \land T,\qquad v\in [0,1].
    \]
    We use $\bar r$ as our new time scale and correspondingly define the rescaled processes 
    \[
    \bar L_u \coloneqq L_{\bar r_u},\quad \bar\Gamma_u \coloneqq \Gamma_{\bar r_u},\quad \bar\xi_u \coloneqq \xi_{\bar r_u},\qquad u\in [0,1].
    \]
    
    To show that the sequence $(\P^n_{(\bar\xi,\bar r)})_n$ is relatively compact in $\c P_2(C([0,1];\R^l\times [0,T]))$ we first note that the monotonicity of $\xi$ implies that
    \begin{align}
        |\bar r_u - \bar r_v| + |\arctan(\Var(\xi, [0,\bar r_u])) - \arctan(\Var(\xi, [0,\bar r_v]))|
        \leq (T + \tfrac \pi 2) |u-v|,\quad u,v\in [0,1].
    \end{align}
    Since 
    \begin{align}
        |\bar\xi_u - \bar\xi_v| \leq |\Var(\xi,[0,\bar r_u]) - \Var(\xi,[0,\bar r_v])|,\qquad u,v\in [0,1],
    \end{align}
    and using that the $\arctan$ function is uniformly continuous on compact intervals, we obtain the desired relative compactness of from the Arzelà-Ascoli theorem along with condition \eqref{eq lemma 3.4 uniform q estimate}. 
    Using that $\norm{\gamma} < \infty$ we see that therefore $(\P^n_{(\bar\Gamma,\bar\xi,\bar r)})_n$ is relatively compact in $\c P_2(C([0,1];\R^d\times\R^l\times [0,T]))$.

    We now transfer the relative compactness result back to the unparametrised measure flow $(\P^n_{(\Gamma,\xi)})_n$ using the mapping $\c S$ introduced in \cref{definition unparametrisation map s}. Since $\c S$ is continuous by \cref{lemma S continuous} it suffices to show that the closure $\overline{(\P^n_{(\bar\Gamma,\bar\xi,\bar r)})_n}$ is still supported on the domain of $\c S$, more precisely that 
   \[ 
        \mu(\c D(\c S)) = 1 \quad \mbox{for all} \quad \P\in \overline{(\P^n_{(\bar\Gamma,\bar\xi,\bar r)})_n}.
\]
    For this we note that for every accumulation point $\P_{(\bar\Gamma,\bar\xi,\bar r)}$ we can find due to the relative compactness of $(\P^n_{\bar L})_n$ a subsequence $\P^{n_k}_{(\bar L,\bar\Gamma,\bar\xi,\bar r)} \to \P_{(\bar L,\bar\Gamma,\bar\xi,\bar r)}$ in $\c P_2(C)$. Now using that by construction for all $n\in\N$,
    \begin{align}
        \bar\Gamma_u = \int_0^u \gamma(\bar r_v,x_{0-} + \bar L_v + \bar\Gamma_v,\bar\xi_v) d\bar\xi_v,\qquad u\in [0,1],\qquad \P^n\text{-a.s.},
    \end{align}
    the convergence ensures due to \autocite[Theorem 7.10]{Kurtz1996} this relation also holds under $\P$. Thus \cref{assumptions gamma montone} ensure that $\P$-a.s.\@ that $\bar\Gamma$ is monotone on every interval where $\bar r$ is constant and therefore, 
\[    
   (\bar\Gamma,\bar\xi,\bar r)\in\c D(\c S) \quad \mbox{$\P$-a.s.}
\]

    The continuity of $\c S$ by \cref{lemma S continuous} now implies the relative compactness of the sequence 
 \[   
    (\c S\#\P^n_{(\bar\Gamma,\bar\xi,\bar r)})_n  = (\P^n_{(\Gamma,\xi)})_n. 
    \]
Since $X = x_{0-} + L + \Gamma$ this shows that $(\mu^n)_n = (\P^n_{(X,\xi)})_n \subseteq \c P_2(D^0)$ is relatively compact.
\end{proof}

\begin{theorem}\label{theorem mean-field game parametrisations accumulation nash equilibrium}
    Let \cref{assumptions state dynamics,assumptions reward and cost functions}, hold and let $\mu$ be an accumulation point of a sequence of Nash equilibria $(\mu^n)_n$ of the $n$-bounded velocity mean-field games \eqref{eq mean-field game regular controls}. Then there exists a parametrisation $\nu$ of $\mu$ such that $\nu$ is a Nash equilibria of the mean-field game \eqref{eq mean-field game parametrisations general definition} of parametrisations.
\end{theorem}

\begin{proof}
    Let us assume that $\mu^n\to\mu$ in $\c P_2(D^0)$. Since $\mu^n$ is a Nash equilibria of the bounded velocity MFG we know that $\mu^n\in\c C(\mu^n)$.
    Thus, by \cref{lemma wm1 convergence parametrisation convergence} we further obtain a sequence of parametrisations $(\nu^n)_n$ with $\nu^n\in \bar{\c A}(\mu^n)$ and a parametrisation $\nu\in\bar{\c A}(\mu)$ of $\mu$ such that, along a suitable subsequence, 
    \[
    \nu^n\to\nu \quad \mbox{in}  \quad \c P_2(C). 
    \]    
        
    Let us w.l.o.g.\@ assume in the following that $\nu^n\to\nu$ in $\c P_2(C)$ and let $\nu'\in\bar{\c A}(\mu)$ be another admissible response to the measure flow $\mu$. If $\nu'_{\bar\xi}\notin\c P_p(C)$, then $J(\mu,\nu') = - \infty$ and thus $J(\mu,\nu) \geq J(\mu,\nu')$. 
    
    For $\nu'_{\bar\xi}\in\c P_p(C)$ it follows from \cref{lemma parametrisation construct lipschitz approximations} that there exists an increasing sequence $(n_k)_k\subseteq\N$ and parametrisations $\hat\nu^k\in \bar{\c A}(\mu^{n_k})$ such that $\hat\nu^k\to \nu'$ in $\c P_2(C)$ and $\hat\nu^k_{\bar\xi} \to \nu_{\bar\xi}'$ in $\c P_p(C)$ and $\hat\mu^k\coloneqq \c S\#\hat\nu^k \in \c A(\mu^{n_k})$ is $n_k$-Lipschitz continuous. Hence $\hat\mu^k$ is an admissible control to the $n_k$-bounded velocity MFG. Since each $\mu^{n_k}$ is a Nash equilibrium to the $n_k$-bounded velocity MFG we know that
    \[
    J(\mu^{n_k},\nu^{n_k}) = J(\mu^{n_k},\mu^{n_k}) \geq J(\mu^{n_k},\hat\mu^k) = J(\mu^{n_k},\hat\nu^k).
    \]
    After taking the limit and using \cref{lemma J parametrisations continuous}, this implies that
    \[
    J(\mu,\nu) \geq \limsup_{k\to\infty} J(\mu^{n_k},\nu^{n_k})
    \geq \lim_{k\to\infty} J(\mu^{n_k},\hat\nu^k) = J(\mu,\nu').
    \]
    Since $\nu'\in\bar{\c A}(\mu)$ was arbitrary, this shows that $\nu$ is indeed a Nash equilibrium for the MFG \eqref{eq mean-field game parametrisations general definition}.
\end{proof}


\section{Conclusion} \label{conclusion}

We established a probabilistic framework for analysing finite-time extended MFGs with multi-dimensional singular controls and state-dependent jump dynamics and costs. Our choice of admissible controls enables an explicit characterisation of the reward function. As the reward function will in general only be u.s.c., we introduced a novel class of MFGs with a broader set of admissible controls, called MFGs of parametrisations. We proved that the reward functional is continuous on the set of parametrisations and established the existence of equilibria, both in MFGs of parametrisations and in the original MFG with singular controls coincide. 

Several avenue are open for future research. First, our focus was on existence of equilibria; no uniqueness of equilibrium results were obtained. In fact, we do not expect our approach to provide a good framework for analysing uniqueness problems. 

Second, we did not consider $N$-player games. Although we strongly expect that the connection between singular controls and parametrisations can also be utilised to prove the existence of equilibria in games with finitely many players, the analysis of $N$-player games and their connections to MFGs is beyond the scope of this paper. 

Third, we considered finite horizon games. MFGs with singular controls on infinite horizons were considered by many authors including \autocite{cao_stationary_2022,dianetti_nonzero-sum_2020,Frederico-etal,Dianetti-Ferrari-Tzouanas}. We strongly expect our results to carry over to infinite-horizon games after suitable modifications of the cost functions the set of admissible controls; see for instance \autocite[Remark 2.12]{dianetti_nonzero-sum_2020} for a ``similar'' extension. The extension to infinite horizon games, too, is left for future research. 


\appendix

\section{Proof of {\cref{lemma wm1 convergence parametrisation convergence}}}\label{appendix proof lemma wm1 convergence parametrisation convergence}

The construction below is based on the proof of \autocite[Theorem 2.8]{denkert2023extended}.

\begin{proof}
    Let us start by setting up our probability space. For any sequence of continuous weak controls $(\mu^n)_n$ with $\mu^n\in \c C(\tilde\mu^n)\subseteq \c A(\tilde\mu^n)$ there exists a sequence of corresponding weak solutions $(\P^n)_n\subseteq \c P_2(\tilde\Omega)$ to our Marcus-type SDE. 
    
    Since $\mu^n\to\mu$ in $\c P_2(D^0)$, the set $(\P^n)_n$ is weakly compact in $\c P_2(\tilde\Omega)$ and thus there exists $\P\in\c P_2(\tilde\Omega)$ such that $\P_{(X,\xi)} = \mu$ and $\P^n\to\P$. Furthermore, since $\tilde\Omega$ is separable, we can use Skorokhod's representation theorem to obtain a joint probability space $(\tilde\Omega,\tilde{\c F},\tilde\P)$ and processes $(X^n,\xi^n,W^n)_n$ and $(X,\xi,W)$ with 
    \[
    \tilde\P_{(X^n,\xi^n,W^n)} = \P^n \quad \mbox{and} \quad \tilde\P_{(X,\xi,W)} = \P
    \] 
    such that
    \begin{align}
        (X^n,\xi^n,W^n) \to (X,\xi,W)\text{ in }\tilde\Omega,\qquad\tilde\P\text{-a.s.\@ and in }L^2.
    \end{align}
    Since $W^n$ is an $\b F^{X^n,\xi^n,W^n}$-Brownian motion for every $n\in\N$, we know that $W$ is also an $\b F^{X,\xi,W}$-Brownian motion. As next step, we define the auxiliary processes
    \begin{align}
        L^n_{t} \coloneqq \int_0^{t} b(s,\tilde\mu^n_{s},X^n_{s},\xi^n_{s}) ds + \int_0^{t} \sigma(s,\tilde\mu^n_{s},X^n_{s},\xi^n_{s}) dW^n_{s},\qquad t\in [0,T],
    \end{align}
    and similarly
    \begin{align}
        L_{t} \coloneqq \int_0^{t} b(s,\tilde\mu_{s},X_{s},\xi_{s}) ds + \int_0^{t} \sigma(s,\tilde\mu_{s},X_{s},\xi_{s}) dW_{s},\qquad t\in [0,T].
    \end{align}
    Since
    \[
    L^n_{t} = X^n_{t} - x_{0-} - \int_0^{t} \gamma(s,X^n_{s},\xi^n_{s}) d\xi^n_{s},\quad t\in [0,T],
    \]
    and $\norm{\gamma}_\infty <\infty$, the sequence $(L^n)_n$ is also $L^2$-uniform integrable. Together with \autocite[Theorem 7.10]{Kurtz1996} this shows that
    \[
    L^n\to L\text{ in }D^0([0,T];\R^d)\qquad\tilde\P\text{-a.s.\@ and in }L^2.
    \]
    Since $L^n$ and $L$ are continuous, this convergence can be strengthened to 
    \begin{align}\label{eq lemma 2.8 C convergence L}
    L^n\to L\text{ in }C([0,T];\R^d)\qquad\tilde\P\text{-a.s.\@ and in }L^2.
    \end{align}

    As next step, we construct the parametrisations $(\nu_n)_n$ of $(\mu_n)_n$. This construction is similar to the second layer in \autocite[Theorem 2.8]{denkert2023extended}. We start by defining for all $\omega\in\tilde\Omega$ the function
    \[
    r^n_{t}(\omega) \coloneqq \frac{t + \arctan(\Var(\xi^n(\omega),[0,t]))}{T + \frac\pi 2},\quad t\in [0,T].
    \]
    The function $r^n$ is strictly monotone and we denote its generalised inverse by
    \[
    \bar r^n_v(\omega) \coloneqq \inf\{t\in [0,T]|r^n_{t}(\omega) > v\}\land T,\quad v\in [0,1].
    \]
    This function $\bar r^n$ will be our new time change; we define the corresponding processes $\bar X^n,\bar\xi^n,\bar L^n$ as follows:
    \[
    \bar X^n_u \coloneqq X^n_{\bar r^n_u},\quad \bar\xi^n_u \coloneqq \xi^n_{\bar r^n_u},\quad \bar L^n_u\coloneqq L^n_{\bar r^n_u},\quad u\in [0,1].
    \]
    By construction $\nu^n\coloneqq \tilde\P_{(\bar X^n,\bar\xi^n,\bar r^n)}$ is a parametrisation of $\mu^n$ and $\nu^n\in\bar{\c A}(\tilde\mu^n)$ since $\mu^n\in\c A(\tilde\mu^n)$. Having defined the parametrisations $(\nu^n)_n$ we now construct their limit along a suitable subsequence. We start by proving that the sequence $$(\tilde\P_{(\bar\xi^n,\bar r^n)})_n\subseteq \c P_2(C([0,1];\R^l\times [0,T]))$$ is relatively compact. Due to the monotonicity of the process $\xi^n$ we have that
    \begin{align}
        |\bar r^n_u - \bar r^n_v| \leq (T+\tfrac{\pi}{2}) |u-v|,\quad u,v\in [0,1],
    \end{align}
    and
    \begin{align}
        |\arctan(\Var(\xi^n,[0,\bar r^n_u]))-\arctan(\Var( \xi^n,[0,\bar r^n_v]))|\leq (T+\tfrac\pi 2)|u-v|,\quad u,v\in [0,1].
    \end{align}
    Using that the $\arctan$ function is uniformly continuous on compact intervals together with the estimate
    \begin{align}
        |\bar\xi^n_u - \bar\xi^n_v| \leq |\Var(\xi^n,[0,\bar r^n_u])-\Var(\xi^n,[0,\bar r^n_v])|,\quad u,v\in [0,1],
    \end{align}
    allows us to apply the Arzelà–Ascoli theorem to show the relative compactness of the sequence $(\tilde\P_{(\bar\xi^n,\bar r^n)})_n$. 

    Using Skorokhod's representation theorem again we can assume that the sequence $(\bar\xi^n,\bar r^n)_n$ also has a convergent subsequence in $L^2$ along which
    \[
    (\bar\xi^n,\bar r^n) \to (\bar\xi,\bar r)\text{ in }C([0,1];\R^l\times [0,T])\quad\tilde\P\text{-a.s.\@ and in }L^2.
    \]
    
    To define the state process $\bar X$, we first use \eqref{eq lemma 2.8 C convergence L} to introduce the process $\bar L$ as follows:
    \[
    \bar L^n = L^n_{\bar r^n} \to L_{\bar r} \eqqcolon \bar L\text{ in }C([0,1];\R^l)\quad\tilde\P\text{-a.s.\@ and in }L^2.
    \]
    Next, we introduce the process
    \[
        \bar\Gamma^n_u
        \coloneqq \int_0^u \gamma(\bar r^n_v,\bar X^n_v,\bar\xi^n_v) d\bar\xi^n_v,\quad u\in [0,1].
     \]
     Since $\nu_n\in\bar{\c A}(\tilde\mu^n)$ the quadruple $(\bar X^n,\bar\xi^n,\bar r^n,W^n)$ satisfies the SDE \eqref{eq weak solution parametrisation SDE}. Thus, $\bar X^n = x_{0-} + \bar L^n + \bar\Gamma^n$ and  
      \begin{align}\label{eq lemma 2.8 bar gamma n definition}
       \bar\Gamma^n_u = \int_0^u \gamma(\bar r^n_v,x_{0-} + \bar L^n_v + \bar\Gamma^n_v,\bar\xi^n_v) d\bar\xi^n_v,\quad u\in [0,1].
    \end{align}
    
    Using that the sequence $(\tilde\P_{\bar\xi^n})_n\subseteq\c P_2(C([0,1];\R^l))$ is relatively compact and that $\norm{\gamma}_\infty<\infty$, we again conclude from the Arzelà-Ascoli theorem that the sequence $(\tilde\P_{\bar\Gamma^n})_n$ is also relatively compact in $\c P_2(C([0,1];\R^d))$. Thus, using Skorokhod's representation theorem we may assume w.l.o.g.~that there exists a function $\bar\Gamma$ such that
    \[
    \bar\Gamma^n\to \bar\Gamma\text{ in }C([0,1];\R^d)\quad\tilde\P\text{-a.s.\@ and in }L^2.
    \]
    Using \eqref{eq lemma 2.8 bar gamma n definition} together with \autocite[Theorem 7.10]{Kurtz1996}, we see that
    \begin{align}\label{eq lemma 2.8 bar gamma sde}
        \bar\Gamma = \lim_{n\to\infty} \bar\Gamma^n = \lim_{n\to\infty} \int_0^\cdot \gamma(\bar r^n_v,x_{0-} + \bar L^n_v + \bar\Gamma^n_v,\bar\xi^n_v) d\bar\xi^n_v = \int_0^\cdot \gamma(\bar r_v,x_{0-} + \bar L_v + \bar\Gamma_v,\bar\xi_v) d\bar\xi_v.
    \end{align}
    Thus we can define our limit state process $\bar X$ as the limit of $(\bar X^n)_n$ as follows
    \[
    \bar X^n = x_{0-} + \bar L^n + \bar\Gamma^n\to x_{0-} + \bar L + \bar\Gamma \eqqcolon \bar X\text{ in }C([0,1];\R^d)\quad\tilde\P\text{-a.s.\@ and in }L^2.
    \]

    While we have shown already that the sequence of parametrisations $(\nu^n)_n$ converges to a limit $\nu\coloneqq\tilde\P_{(\bar X,\bar\xi,\bar r)}$ along a subsequence, it remains to show that $\nu\in\bar{\c A}(\tilde\mu)$ and that $\nu$ is a parametrisation of $\mu$.

    We first show that $(\bar X,\bar\xi,\bar r)\in \c D(\c S)$, $\tilde\P$-a.s. We recall that $\bar L \coloneqq L_{\bar r}$, which implies that on every interval where $\bar r$ is constant, $\bar L$ is constant too. Moreover, $\bar\Gamma$ is monotone on each such interval by \cref{assumptions gamma montone}. 
    Due to $\bar X = x_{0-} + \bar L + \bar\Gamma$ we conclude that $(\bar X,\bar\xi,\bar r)\in\c D(\c S)$, $\tilde\P$-a.s. We can now use the continuity of the mapping $\c S$ established in \cref{lemma S continuous} to obtain that
    \begin{align}
    S(\bar X,\bar\xi,\bar r) = \lim_{n\to\infty} \c S(\bar X^n,\bar\xi^n,\bar r^n) = \lim_{n\to\infty} (X^n,\xi^n) = (X,\xi).
    \end{align}
    This implies that $\mu = \c S\#\nu$. Now plugging this into the definition of $L$, we obtain 
    \begin{align}
    L_{t} &= \int_0^{t} b(s,\tilde\mu_{s},X_{s},\xi_{s}) ds + \int_0^{t} \sigma(s,\tilde\mu_{s},X_{s},\xi_{s}) dW_{s}\\
    &= \int_0^{t} b(s,\tilde\mu_{s},\c S(\bar X,\bar\xi,\bar r)_{s}) ds + \int_0^{t} \sigma(s,\tilde\mu_{s},\c S(\bar X,\bar\xi,\bar r)_{s}) dW_{s}\\
    &= \int_0^{t} b(s,\tilde\mu_{s},\bar X_{r_{s}},\bar\xi_{r_{s}}) ds + \int_0^{t} \sigma(s,\tilde\mu_{s},\bar X_{r_{s}},\bar\xi_{r_{s}}) dW_{s},\quad t\in [0,T].
    \end{align}
    Therefore, we arrive at
    \[
    \bar L_u = L_{\bar r_u} = \int_0^u b(\bar r_v,\tilde\mu_{\bar r_v},\bar X_v,\bar\xi_v) d\bar r_v + \int_0^u \sigma(\bar r_v,\tilde\mu_{\bar r_v},\bar X_v,\bar\xi_v) dW_{\bar r_v},\quad u\in [0,1].
    \]
    Together with \eqref{eq lemma 2.8 bar gamma sde} shows that $(\bar X,\bar\xi,\bar r)$ indeed satisfies the SDE \eqref{eq weak solution parametrisation SDE}. Further, since every $W^n$ is an $(\c F^W_{t}\lor \c F^{\bar X^n,\bar\xi^n,\bar r^n}_{r^n_{t}})_{t\in [0,T]}$-Brownian motion, their limit $W$ is an $(\c F^W_{t}\lor \c F^{\bar X,\bar\xi,\bar r}_{r_{t}})_{t\in [0,T]}$-Brownian motion and thus $$\nu\coloneqq \tilde\P_{(\bar X,\bar\xi,\bar r)}\in\bar{\c A}(\tilde\mu).$$ 
\end{proof}


\section{Proof of {\cref{lemma parametrisation construct lipschitz approximations}}}\label{appendix proof lemma parametrisation construct lipschitz approximations}
The proof follows from the following three lemmas. The first two lemmas are essentially simpler versions of \autocite[Appendix B]{denkert2023extended}. We include a detailed proofs to keep the paper self-contained.

\begin{lemma}\label{lemma a1}
    Let \cref{assumptions state dynamics} hold.
    Let $\tilde\mu\in\c P_2(D^0)$ be a given measure flow and $\nu\in\bar{\c A}(\tilde\mu)$ be a parametrisation of $\mu\in \c P_2(D^0)$. Then there exists a sequence of Lipschitz continuous parametrisations $(\nu^n)_n\subseteq\bar{\c A}(\tilde\mu)$ of $(\mu^n)_n \coloneqq (\c S\#\nu^n)_n\subseteq \c P_2(D^0)$, such that
    \[
    (\mu^n,\nu^n) \to (\mu,\nu)\qquad\text{ in }\c P_2(D^0)\times \c P_2(C).
    \]
    If $(\mu_\xi,\nu_{\bar\xi}) \in \c P_q(D^0([0,T];\R^l))\times\c P_q(C([0,1];\R^l))$ for some $q > 2$, then we can choose $(\mu^n,\nu^n)_n$ such that 
    \[
    (\mu^n_\xi,\nu^n_{\bar\xi}) \to (\mu_\xi,\nu_{\bar\xi})\quad\text{ in } \quad \c P_q(D^0([0,T];\R^l))\times\c P_q(C([0,1];\R^l)).
    \]
\end{lemma}

\begin{proof}
    This proof is based on the construction of \autocite[Lemma B.1]{denkert2023extended}. Since our parametrisations only involve a single layer, the proof will be simpler. We construct the desired parametrisations in two steps. 

    \begin{enumerate}[wide]
        \item[\textbf{Step 1. \emph{Truncating the given parametrisation.}}]
        We are given the measure flow $\tilde\mu\in \c P_2(D^0)$ and a parametrisation $\nu\in\bar{\c A}(\tilde\mu)$ of $\mu\in\c P_2(D^0)$. By definition there exists a probability measure $\bar\P\in\c P_2(\bar\Omega)$ such that
        \begin{align}
            d\bar X_u &= b(\bar r_u,\tilde\mu_{\bar r_u},\bar X_u,\bar\xi_u) d\bar r_u + \sigma(\bar r_u,\tilde\mu_{\bar r_u},\bar X_u,\bar\xi_u) d\bar W_{\bar r_u} + \gamma(\bar r_u,\bar X_u,\bar\xi_u) d\bar\xi_u,\quad u\in [0,1],\quad \bar X_0 = x_{0-}.
            \label{eq lemma a1 original sdes}
        \end{align}
        For every $K>0$, we introduce the truncated control process
        \[
        \bar\xi^K_u \coloneqq \bar\xi_u\land K,\quad u\in [0,1]
        \]
        and consider the corresponding state processes
        \begin{align}
            d\bar X^K_u &= b(\bar r_u,\tilde\mu_{\bar r_u},\bar X^K_u,\bar\xi^K_u) d\bar r_u + \sigma(\bar r_u,\tilde\mu_{\bar r_u},\bar X^K_u,\bar\xi^K_u) d\bar W_{\bar r_u} + \gamma(\bar r_u,\bar X^K_u,\bar\xi^K_u) d\bar\xi^K_u,\quad u\in [0,1],\quad \bar X^K_0 = x_{0-},
        \end{align}
        along with the parametrisations $\nu^K\coloneqq \bar\P_{(\bar X^K,\bar\xi^K,\bar r)}\in\bar{\c A}(\tilde\mu)$ of $\mu^K \coloneqq \c S\#\nu^K$. By a standard Gronwall argument, $(\bar X^K,\bar\xi^K,\bar r) \to (\bar X,\bar\xi,\bar r)$ in $L^2$ and thus $\nu^K\to\nu$ in $\c P_2(C)$ as $K\to\infty$. Since $\nu$ is a parametrisation of $\mu$ and $\nu^K$ is a parametrisation of $\mu^K$ it follows from \cref{lemma S continuous} that
        \[
            \mu^K\to\mu \quad \mbox{in} \quad \c P_2(D^0).
        \]

        \item[\textbf{Step 2. \emph{Approximation with Lipschitz parametrisations.}}]
        Let us now fix $K$ and reparametrise the parametrisation $\nu^K$ of $\mu^K$ to be Lipschitz continuous. To this end, we introduce for every $\varepsilon>0$ the time change
        \[
        \beta^{K,\varepsilon}_u(\omega) \coloneqq \frac{u+\varepsilon\bigr(\bar r_u(\omega) + \Var(\bar\xi^K(\omega),[0,u])\bigr)}{1+\varepsilon(T+lK)},\qquad u\in [0,1],
        \]
        where $\Var$ denotes the total Variation. We denote the generalised inverse time change by $\bar\beta^{K,\varepsilon}_v(\omega) \coloneqq \inf\{u\in [0,1]\mid \beta^{K,\varepsilon}_u > v\}\land 1$ and define the following reparametrised processes:
        \[
        (\bar X^{K,\varepsilon}_u,\bar\xi^{K,\varepsilon}_u,\bar r^{K,\varepsilon}_u) \coloneqq (\bar X^K_{\bar\beta^{K,\varepsilon}_u},\bar\xi^K_{\bar\beta^{K,\varepsilon}_u},\bar r_{\bar\beta^{K,\varepsilon}_u}),\qquad u\in [0,1].
        \]
        By construction $(\bar\xi^{K,\varepsilon},\bar r^{K,\varepsilon})$ is Lipschitz continuous: due to the monotonicity of $\bar\xi^K,\bar r$ and $\bar\beta^{K,\varepsilon}$ it holds for any $0\leq u\leq v\leq 1$ that
        \begin{align}
            |\bar\xi^{K,\varepsilon}_v - \bar\xi^{K,\varepsilon}_u| + |\bar r^{K,\varepsilon}_v - \bar r^{K,\varepsilon}_u|
            &\leq \Bigl|\frac 1 \varepsilon(\bar\beta^{K,\varepsilon}_v - \bar\beta^{K,\varepsilon}_u) + \bar r_{\bar\beta^{K,\varepsilon}_v} - \bar r_{\bar\beta^{K,\varepsilon}_u} + \Var(\bar\xi^K,[\bar\beta^{K,\varepsilon}_u,\bar\beta^{K,\varepsilon}_v])\Bigr| \\
            &\leq \frac{1+\varepsilon(T+lK)}\varepsilon |v-u|.
            \label{eq lemma a1 lipschitz continuity epsilon rescaling}
        \end{align}
        Further, we also note that for $u\in [0,1]$, using $\beta^{K,\varepsilon}_{\bar\beta^{K,\varepsilon}_u} = u\land \beta^{K,\varepsilon}_1$,
        \begin{align}
            |\bar\beta^{K,\varepsilon}_u - u|
            &\leq |\bar\beta^{K,\varepsilon}_u - \beta^{K,\varepsilon}_{\bar\beta^{K,\varepsilon}_u}| + |u\land \beta^{K,\varepsilon}_1 - u|\\
            &\leq \Bigl|\bar\beta^{K,\varepsilon}_u - \frac{\bar\beta^{K,\varepsilon}_u+\varepsilon\bigr(\bar r_{\bar\beta^{K,\varepsilon}_u}(\omega) + \Var(\bar\xi^K(\omega),[0,\bar\beta^{K,\varepsilon}_u])\bigr)}{1+\varepsilon(T+lK)} \Bigr| + \Bigl|1 - \frac{1+\varepsilon\bigr(T + \Var(\bar\xi^K(\omega),[0,1])\bigr)}{1+\varepsilon(T+lK)} \Bigr| \\
            &\leq 2\varepsilon(T+lK).
            \label{eq lemma a1 convergence epsilon rescaling}
        \end{align}
        Together with $(\bar X^K,\bar\xi^K,\bar r)_K$ being $\bar\P$-a.s.\@ uniformly continuous, this implies that
        \[
        (\bar X^{K,\varepsilon},\bar\xi^{K,\varepsilon},\bar r^{K,\varepsilon}) \to (\bar X^K,\bar\xi^K,\bar r)\text{ in } C([0,1];\R^d\times\R^l\times [0,T])\qquad\bar\P\text{-a.s.},\qquad\text{ as }\varepsilon\to 0,
        \]
        and thus
        \[
        \nu^{K,\varepsilon}\coloneqq \bar\P_{(\bar X^{K,\varepsilon},\bar\xi^{K,\varepsilon},\bar r^{K,\varepsilon})} \to \nu^K,\quad\text{in} \quad \c P_2(C([0,1];\R^d\times\R^l\times [0,T])).
        \]
        In view of \cref{lemma S continuous}, since by construction $\nu^{K,\varepsilon}\in\bar{\c A}(\tilde\mu)$ this implies that 
        \[
            \mu^{K,\varepsilon}\coloneqq \c S\#\nu^{K,\varepsilon}\to\mu^K \quad \mbox{in} \quad \c P_2(D^0).
        \] 

        Finally, by choosing a suitable subsequence such $\varepsilon\to 0$ fast enough while $K\to\infty$, we obtain the desired sequence. If $(\mu_\xi,\nu_{\bar\xi}) \in \c P_q(D^0([0,T];\R^l))\times\c P_q(C([0,1];\R^l))$, then this construction also satisfies that 
        \[
            (\mu^{K,\varepsilon}_\xi,\nu^{K,\varepsilon}_{\bar\xi}) \to (\mu_\xi,\nu_{\bar\xi}) \quad \mbox{in} \quad \c P_q(D^0([0,T];\R^l))\times\c P_q(C([0,1];\R^l)).
        \]
    \end{enumerate}
\end{proof}

\begin{lemma}\label{lemma a1.5}
    Let \cref{assumptions state dynamics} hold.
    Let $\tilde\mu\in\c P_2(D^0)$ be a given measure flow and $\nu\in\bar{\c A}(\tilde\mu)$ be a Lipschitz continuous parametrisation of $\mu\in \c P_2(D^0)$. Then there exists a sequence of Lipschitz continuous controls $(\mu^n)_n\subseteq \c A(\tilde\mu)$ with Lipschitz continuous parametrisations $(\nu^n)_n\subseteq\bar{\c A}(\tilde\mu)$ such that
    \[
    (\mu^n,\nu^n) \to (\mu,\nu)\qquad\text{ in }\c P_2(D^0)\times \c P_2(C).
    \]
    Further if $(\mu_\xi,\nu_{\bar\xi}) \in \c P_q(D^0([0,T];\R^l))\times\c P_q(C([0,1];\R^l))$ for some $q > 2$, then we can choose $(\mu^n,\nu^n)_n$ such that additionally
    \[
    (\mu^n_\xi,\nu^n_{\bar\xi}) \to (\mu_\xi,\nu_{\bar\xi})\qquad\text{ in }\c P_q(D^0([0,T];\R^l))\times\c P_q(C([0,1];\R^l)).
    \]
\end{lemma}

\begin{proof}
    This proof is based on the construction in \autocite[Lemma B.2]{denkert2023extended}, albeit again in a simpler single layer setting. We proceed in three steps. 
    
    \begin{enumerate}[wide]
        \item[\textbf{Step 1. \emph{The control process.}}]
        Since $\nu\in\bar{\c A}(\tilde\mu)$ is a parametrisation, there exists a probability measure $\bar\P\in\c P_2(\bar\Omega)$ such that
        \begin{align}
            d\bar X_u &= b(\bar r_u,\tilde\mu_{\bar r_u},\bar X_u,\bar\xi_u) d\bar r_u + \sigma(\bar r_u,\tilde\mu_{\bar r_u},\bar X_u,\bar\xi_u) d\bar W_{\bar r_u} + \gamma(\bar r_u,\bar X_u,\bar\xi_u) d\bar\xi_u,\quad u\in [0,1],\quad \bar X_0 = x_{0-}.
            \label{eq lemma a1.5 original sdes}
        \end{align}
        
        Since $\nu$ is a Lipschitz parametrisation, the processes $\bar\xi$ and $\bar r$ are Lipschitz continuous. The desired approximating parametrisations $(\mu^\delta)_\delta$ will be obtained by only perturbing the reparametrised time scale $\bar r$ such that the resulting time scale $\bar r^\delta$ and its inverse $r^\delta$ are Lipschitz continuous, and then working with the unparametrised Lipschitz continuous control processes $\xi^\delta \coloneqq \bar\xi_{r^\delta}$. 

        To this end, we introduce, for every $\delta > 0$, the perturbed reparametrised time
        \[
        \bar r^\delta_u \coloneqq \frac{\bar r_u+\delta T u}{1+\delta},\qquad u\in [0,1].
        \]
        We denote its inverse by $r^\delta$ and introduce the unparametrised control processes
        \[
        \xi^\delta_{t} \coloneqq \bar\xi_{r^\delta_{t}},\qquad t\in [0,T].
        \]

        The Lipschitz continuity of $\bar r$ implies the Lipschitz continuity of $\bar r^\delta$. Furthermore, the function $r^\delta$ is Lipschitz continuous since for any $0\leq s\leq t\leq T$, by montonicity of $\bar r$ and $r^\delta$,
        \begin{align}
            |r^\delta_{t} - r^\delta_{s}|
            \leq \Bigl|\frac 1 {\delta T} (\bar r_{r^\delta_{t}} - \bar r_{r^\delta_{s}}) + (r^\delta_{t} - r^\delta_{s}) \Bigr|
            =\frac{1+\delta}{\delta T} |t-s|.
        \end{align}
        Since $\bar\xi$ is Lipschitz continuous uniformly in $\omega$ with some Lipschitz constant $C_{\bar\xi}$, this implies that the control process $\xi^\delta$ is Lipschitz continuous uniformly in $\omega$ as well: for $0\leq s\leq t\leq T$,
        \[
        |\xi^\delta_{t} - \xi^\delta_{s}|
        = |\bar\xi_{r^\delta_{t}} - \bar\xi_{r^\delta_{s}}|
        \leq C_{\bar\xi} |r^\delta_{t} - r^\delta_{s}|.
        \]
    
        \item[\textbf{Step 2. \emph{The state process.}}]
        To guarantee the convergence of the corresponding state processes, we use the fact that we are working weak solutions and are thus allowed to choose the Brownian motion. Specifically, we define a Brownian motions for $(\bar\xi,\bar r^\delta)$ and $\xi^\delta$ in terms of the Brownian motion $\bar W$ corresponding to the process $(\bar X,\bar\xi,\bar r)$ introduced in \eqref{eq lemma a1.5 original sdes}. 

        We start by rewriting \eqref{eq lemma a1.5 original sdes} by introducing a Brownian motion $\bar B$ on some enlarged probability space $(\hat\Omega,\hat{\c F},\hat\P)$ such that
        \[
        d\bar W_{\bar r_u} = \sqrt{\dot{\bar r}_u} d\bar B_u,\qquad u\in [0,1].
        \]
        This can be achieved by enlarging our probability space to allow for an independent Brownian motion $\hat B$ and then defining
        \[
        \bar B_v \coloneqq \int_0^v \1_{\{\dot{\bar r}_u \not= 0\}} \frac 1 {\sqrt{\dot{\bar r}_u}} d\bar W_{\bar r_u} + \int_0^v \1_{\{\dot{\bar r}_u=0\}} d\hat B_u,\qquad v\in [0,1].
        \]
        Then $\bar B$ is an $\b F^{\bar B,\bar X,\bar\xi,\bar r}$-Brownian motion. Using this Brownian motion, we can now rewrite \eqref{eq lemma a1.5 original sdes} as
        \begin{align}
            d\bar X_u &= b(\bar r_u,\tilde\mu_{\bar r_u},\bar X_u,\bar\xi_u) d\bar r_u + \sigma(\bar r_u,\tilde\mu_{\bar r_u},\bar X_u,\bar\xi_u) \sqrt{\dot{\bar r}_u} d\bar B_u + \gamma(\bar r_u,\bar X_u,\bar\xi_u) d\bar\xi_u,\quad u\in [0,1],\quad \bar X_0 = x_{0-}.
            \label{eq lemma a1.5 rewritten sdes}
        \end{align}
        In terms of $\bar B$ we then define the Brownian motion $\bar W^\delta$ for our new state process in such a way that
        \[
        d\bar W^\delta_{\bar r^\delta_u} = \sqrt{\dot{\bar r}^\delta_u} d\bar B_u,\qquad u\in [0,1],
        \]
        by letting
        \[
        \bar W^\delta_t \coloneqq \int_0^{r^\delta_t} \sqrt{\dot{\bar r}^\delta_u} d\bar B_u,\qquad t\in [0,T].
        \]
        Then $\bar W^\delta$ is an $(\c F^{\bar W^\delta}_{t}\lor \c F^{\bar X,\bar\xi,\bar r}_{r^\delta_{t}})_{t\in [0,T]}$-Brownian motion and hence also an $(\c F^{\bar W^\delta}_{t}\lor \c F^{\bar X,\bar\xi,\bar r^\delta}_{r^\delta_{t}})_{t\in [0,T]}$-Brownian motion as $r^\delta$ is $\b F^{\bar r}$-adapted. We now define $X^\delta$ and $\bar X^\delta$ as the solutions to the following SDEs
        \begin{align}
            d\bar X^\delta_u &= b(\bar r^\delta_u,\tilde\mu_{\bar r^\delta_u},\bar X^\delta_u,\bar\xi_u) d\bar r^\delta_u + \sigma(\bar r^\delta_u,\tilde\mu_{\bar r^\delta_u},\bar X^\delta_u,\bar\xi_u) d\bar W^\delta_{\bar r^\delta}
            + \gamma(\bar r^\delta_u,\bar X^\delta_u,\bar\xi_u) d\bar\xi_u,\quad u\in [0,1],\quad \bar X^\delta_0 = x_{0-},\\
            dX^\delta_t &= b(t,\tilde\mu_t,X^\delta_t,\xi^\delta_t) dt + \sigma(t,\tilde\mu_t,X^\delta_t,\xi^\delta_t) d\bar W^\delta_t
            + \gamma(t,X^\delta_t,\xi^\delta_t) d\xi^\delta_t,\quad t\in [0,T],\quad X^\delta_{0-} = x_{0-}.
            \label{eq lemma a1.5 approximating sdes}
        \end{align}
        
        By construction $\bar W^\delta$ is also an $(\c F^{\bar W^\delta}_{t}\lor \c F^{\bar X^\delta,\bar\xi,\bar r^\delta}_{r^\delta_{t}})_{t\in [0,T}$- and $\b F^{\bar W^\delta,X^\delta,\xi^\delta}$-Brownian motion. The approximating control $\mu^\delta\coloneqq \hat\P_{(X^\delta,\xi^\delta)}$ thus belongs to $\c A(\tilde\mu)$ and $\nu^\delta\coloneqq \hat\P_{(\bar X^\delta,\bar\xi,\bar r^\delta)} \in \bar{\c A}(\tilde\mu)$ is a parametrisation of $\mu^\delta$.

        \item[\textbf{Step 3. \emph{Verification.}}]
        We have seen that the control $\mu^\delta\in \c A(\tilde\mu)$ and its parametrisation $\nu^\delta\in\bar{\c A}(\tilde\mu)$ are both Lipschitz continuous. It remains to show that $(\mu^\delta,\nu^\delta) \to (\mu,\nu)$ in $\c P_2(D^0)\times\c P_2(C)$ as $\delta\to 0$.
        
        We first focus on the convergence of the processes $(\bar X^\delta,\bar\xi,\bar r^\delta)$, which then implies the convergence of the parametrisations $\nu^\delta$. We start with the convergence of the time scales $\bar r^\delta$. For $u\in [0,1]$,
        \begin{align}
            |\bar r^\delta_u - \bar r_u|
            &\leq \frac \delta {1+\delta} |T u - \bar r_u|\leq \delta T,
        \end{align}
        which implies that $\bar r^\delta \to \bar r$ uniformly in $u,\omega$ as $\delta\to 0$. Furthermore,
        \begin{align}
            |\dot{\bar r}^\delta_u - \dot{\bar r}_u| \leq \frac {\delta \dot{\bar r}_u + \delta T} {1+\delta}  \leq  \delta (C_{\bar r} + T) \to 0,\qquad\text{as }\delta\to 0,
            \label{eq lemma a2 derivative r convergence}
        \end{align}
        where $C_{\bar r}$ is the Lipschitz constant of $\bar r$ uniform in $\omega\in\hat\Omega$.
        
        To prove the convergence of the state process $\bar X^\delta$ we rely on a Gronwall-type argument. We start by estimating for $u_*\in [0,1]$,
        \begin{align}
            \E^{\hat\P}\Bigl[\sup_{u\in [0,u_*]} |\bar X_u - \bar X^\delta_u|^2\Bigr]
            &\leq 4(I_1 + I_2 + I_3),
        \end{align}
        where
        \begin{align}
            I_1 &= \E^{\hat\P}\Bigl[\sup_{u\in [0,u_*]} \Bigl|\int_0^u b(\bar r_v,\tilde\mu_{\bar r_v},\bar X_v,\bar\xi_v) d\bar r_v
            - \int_0^u b(\bar r^\delta_v,\tilde\mu_{\bar r^\delta_v},\bar X^\delta_v,\bar\xi_v) d\bar r^\delta_v \Bigr|^2 \Bigr], \\
            I_2 &= \E^{\hat\P}\Bigl[\sup_{u\in [0,u_*]} \Bigl|\int_0^u \sigma(\bar r_v,\tilde\mu_{\bar r_v},\bar X_v,\bar\xi_v) d\bar W_{\bar r_v}
            - \int_0^u \sigma(\bar r^\delta_v,\tilde\mu_{\bar r^\delta_v},\bar X^\delta_v,\bar\xi_v) d\bar W^\delta_{\bar r^\delta_v} \Bigr|^2 \Bigr] \\
            &= \E^{\hat\P}\Bigl[\sup_{u\in [0,u_*]} \Bigl|\int_0^u \sigma(\bar r_v,\tilde\mu_{\bar r_v},\bar X_v,\bar\xi_v)\sqrt{\dot{\bar r}_v} d\bar B_v
            - \int_0^u \sigma(\bar r^\delta_v,\tilde\mu_{\bar r^\delta_v},\bar X^\delta_v,\bar\xi_v) \sqrt{\dot{\bar r}^\delta_v} d\bar B_v \Bigr|^2 \Bigr], \\
            I_3 &= \E^{\hat\P}\Bigl[\sup_{u\in [0,u_*]} \Bigl|\int_0^u \gamma(\bar r_v,\bar X_v,\bar\xi_v) d\bar \xi_v
            - \int_0^u \gamma(\bar r^\delta_v,\bar X^\delta_v,\bar\xi_v) d\bar \xi_v \Bigr|^2 \Bigr] .
        \end{align}
        Bounding the first two terms $I_1$ and $I_2$ from above is standard using the Lipschitz continuity of $b,\sigma$ and $\bar r$ together with the uniform convergence in \eqref{eq lemma a2 derivative r convergence}.
        
        To bound the third term $I_3$, we observe that, using the monotonicity and Lipschitz continuity of $\bar\xi$,
        \begin{align}
            I_3 &\leq C_{\bar\xi}^2 \E^{\hat\P}\Bigl[\int_0^{u_*} \bigl|\gamma(\bar r_v,\bar X_v,\bar\xi_v) 
            - \gamma(\bar r^\delta_v,\bar X^\delta_v,\bar\xi_v) \bigr|^2 dv \Bigr]\\
            &\leq 2C_{\bar\xi}^2 \E^{\hat\P}\Bigl[\int_0^{u_*} \bigl|\gamma(\bar r_v,\bar X_v,\bar\xi_v) 
            - \gamma(\bar r^\delta_v,\bar X_v,\bar\xi_v) \bigr|^2 dv \Bigr]
            + 2 C_{\bar\xi}^2 \E^{\hat\P}\Bigl[\int_0^{u_*} \bigl|\gamma(\bar r^\delta_v,\bar X_v,\bar\xi_v) 
            - \gamma(\bar r^\delta_v,\bar X^\delta_v,\bar\xi_v) \bigr|^2 dv \Bigr]\\
            &\leq 2C_{\bar\xi}^2 \E^{\hat\P}\Bigl[\sup_{v\in [0,1]} \bigl|\gamma(\bar r_v,\bar X_v,\bar\xi_v) 
            - \gamma(\bar r^\delta_v,\bar X_v,\bar\xi_v) \bigr|^2 \Bigr]
            + 2 C_{\bar\xi}^2 C_\gamma^2 \E^{\hat\P}\Bigl[\int_0^{u^*} |\bar X_v - \bar X^\delta_v|^2 dv \Bigr].
        \end{align}
        The first term vanishes due to the continuity of $\gamma$ by dominated convergence. By applying Gronwall's inequality, this implies that
        \begin{align}
            \c W_2^2(\nu^\delta,\nu) \leq \E^{\hat\P}\Bigl[\sup_{u\in [0,1]} |\bar X_u - \bar X^\delta_u|^2\Bigr] + \E^{\hat\P}\Bigl[\sup_{u\in [0,1]} |\bar r_u - \bar r^\delta_u|^2\Bigr] \to 0,\qquad\text{as }\delta\to 0.
        \end{align}
        
        Since $\nu^\delta\to\nu$ in $\c P_2(C)$ and $\nu$ is a parametrisation of $\mu$ and $\nu^\delta$ are parametrisations of $\mu^\delta$, it follows from \cref{lemma S continuous} that $\mu^\delta\to\mu$ in $\c P_2(D^0)$. If $(\mu_\xi,\nu_{\bar\xi}) \in \c P_q(D^0([0,T];\R^l))\times\c P_q(C([0,1];\R^l))$, then we also obtain that
        \[
        (\mu^\delta_\xi,\nu^\delta_{\bar\xi}) \to (\mu_\xi,\nu_{\bar\xi}) \quad \mbox{in} \quad \c P_q(D^0([0,T];\R^l))\times\c P_q(C([0,1];\R^l)).
        \]
    \end{enumerate}
\end{proof}

The final lemma guarantees that we can replace $\mu_n\in \c A(\tilde\mu)$ and $\nu_n\in\bar{\c A}(\tilde\mu)$ by $\mu'_n\in \c A(\tilde\mu_{k_n})$ and $\nu'_n\in\bar{\c A}(\tilde\mu_{k_n})$, respectively.

\begin{lemma}\label{lemma a2}
    Let \cref{assumptions state dynamics} hold.
    Let $\tilde\mu\in\c P_2(D^0)$ be a given measure flow and let $\mu\in\c A(\tilde\mu)$ be a Lipschitz continuous control and $\nu$ a Lipschitz continuous parametrisation of $\mu$. Then for every other measure flow $\tilde\mu'\in\c P_2(D^0)$, there exists a Lipschitz continuous control $\mu'\in\c A(\tilde\mu')$ with the same Lipschitz constant as $\mu$ and a corresponding Lipschitz continuous parametrisation $\nu'\in\bar{\c A}(\tilde\mu')$ with the same Lipschitz constant as $\nu$, such that $\mu'_\xi = \mu_\xi$, $\nu'_{\bar\xi} = \nu_{\bar\xi}$ and
    \[
    \c W_2(\nu,\nu') \leq C_{\nu} \c W_2(\tilde\mu,\tilde\mu'),
    \]
    where the constant $C_{\nu}$ only depends on the Lipschitz constant of $\nu$.
\end{lemma}

\begin{proof}
    Since $\nu\in\bar{\c A}(\tilde\mu)$ is a parametrisation of $\mu\in\c A(\tilde\mu)$, by definition there exists a probability measure $\bar\P\in\c P_2(\bar\Omega)$ such that
    \begin{align}
        d\bar X_u &= b(\bar r_u,\tilde\mu_{\bar r_u},\bar X_u,\bar\xi_u) d\bar r_u + \sigma(\bar r_u,\tilde\mu_{\bar r_u},\bar X_u,\bar\xi_u) d\bar W_{\bar r_u} + \gamma(\bar r_u,\bar X_u,\bar\xi_u) d\bar\xi_u,\quad u\in [0,1],\quad \bar X_0 = x_{0-},\\
        dX_t &= b(t,\tilde\mu_t,X_t,\xi_t) dt + \sigma(t,\tilde\mu_t,X_t,\xi_t) d\bar W_t + \gamma(t,X_t,\xi_t) \diamond d\xi_t,\quad t\in [0,T],\quad X_{0-} = x_{0-}.
        \label{eq lemma a2 original sdes}
    \end{align}
    Now we define the processes $X'$ and $\bar X'$ as the solutions to the following SDEs
    \begin{align}
        d\bar X'_u &= b(\bar r_u,\tilde\mu'_{\bar r_u},\bar X'_u,\bar\xi_u) d\bar r_u + \sigma(\bar r_u,\tilde\mu'_{\bar r_u},\bar X'_u,\bar\xi_u) d\bar W_{\bar r_u} + \gamma(\bar r_u,\bar X'_u,\bar\xi_u) d\bar\xi_u,\quad u\in [0,1],\quad \bar X'_0 = x_{0-},\\
        dX'_t &= b(t,\tilde\mu'_t,X'_t,\xi_t) dt + \sigma(t,\tilde\mu'_t,X'_t,\xi_t) d\bar W_t + \gamma(t,X'_t,\xi_t) \diamond d\xi_t,\quad t\in [0,T],\quad X'_{0-} = x_{0-}.
        \label{eq lemma a2 approximating sdes}
    \end{align}
    We denote the corresponding measures by $\mu'\coloneqq \bar\P_{(X',\xi)}$ and $\nu'\coloneqq \bar\P_{(\bar X',\bar\xi,\bar r)}$. By construction $\mu'\in\c A(\tilde\mu')$ is Lipschitz with the same Lipschitz constant as $\mu$ and $\nu'$ is a parametrisation of $\mu'$ and $\nu'\in\bar{\c A}(\tilde\mu')$ is Lipschitz with the same Lipschitz constant as $\nu$.

    To proof the desired estimate, we will use Gronwall's inequality. We start by noting that for $u_*\in [0,1]$,
    \begin{align}
        \E^{\bar\P}\Bigl[\sup_{u\in [0,u_*]} |\bar X_u - \bar X'_u|^2\Bigr]
        &\leq 4(I_1 + I_2 + I_3),
    \end{align}
    where
    \begin{align}
        I_1 &= \E^{\bar\P}\Bigl[\sup_{u\in [0,u_*]} \Bigl|\int_0^u b(\bar r_v,\tilde\mu_{\bar r_v},\bar X_v,\bar\xi_v) d\bar r_v
        - \int_0^u b(\bar r_v,\tilde\mu'_{\bar r_v},\bar X'_v,\bar\xi_v) d\bar r_v \Bigr|^2 \Bigr], \\
        I_2 &= \E^{\bar\P}\Bigl[\sup_{u\in [0,u_*]} \Bigl|\int_0^u \sigma(\bar r_v,\tilde\mu_{\bar r_v},\bar X_v,\bar\xi_v) d\bar W_{\bar r_v}
        - \int_0^u \sigma(\bar r_v,\tilde\mu'_{\bar r_v},\bar X'_v,\bar\xi_v) d\bar W_{\bar r_v} \Bigr|^2 \Bigr], \\
        I_3 &= \E^{\bar\P}\Bigl[\sup_{u\in [0,u_*]} \Bigl|\int_0^u \gamma(\bar r_v,\bar X_v,\bar\xi_v) d\bar \xi_v
        - \int_0^u \gamma(\bar r_v,\bar X'_v,\bar\xi_v) d\bar \xi_v \Bigr|^2 \Bigr] .
    \end{align}
    Then for the first term $I_1$, we note that
    \begin{align}
        I_1
        &\leq \E^{\bar\P}\Bigl[\int_0^{u_*} \bigl|b(\bar r_v,\tilde\mu_{\bar r_v},\bar X_v,\bar\xi_v) - b(\bar r_v,\tilde\mu'_{\bar r_v},\bar X'_v,\bar\xi_v) \bigr|^2 \dot{\bar r}_v d\bar r_v\Bigr]\\
        &\leq 2 C_{\bar r} C_b^2 \E^{\bar\P}\Bigl[\int_0^{u_*} \c W_2^2(\tilde\mu_{\bar r_v},\tilde\mu'_{\bar r_v}) + |\bar X_v - \bar X'_v|^2 d\bar r_v\Bigr]\\
        &\leq 2  C_{\bar r}^2 C_b^2 \int_0^{u_*} \E^{\bar\P}\Bigl[\sup_{v\in [0,u]} |\bar X_v - \bar X'_v|^2 \Bigr] du + 2  C_{\bar r} C_b^2 \int_0^T \c W_2^2(\tilde\mu_t,\tilde\mu'_t) dt.
    \end{align}
    Similarly, for the second term $I_2$,
    \begin{align}
        I_2
        &\leq \E^{\bar\P}\Bigl[\sup_{u\in [0,u_*]} \Bigl|\int_0^u \bigl(\sigma(\bar r_v,\tilde\mu_{\bar r_v},\bar X_v,\bar\xi_v) - \sigma(\bar r_v,\tilde\mu'_{\bar r_v},\bar X'_v,\bar\xi_v) \bigr)\sqrt{\dot{\bar r}_v} d\bar W_{\bar r_v} \Bigr|^2\Bigr]\\
        &\leq 8 C_\sigma^2 \E^{\bar\P}\Bigl[\int_0^{u_*} \c W_2^2(\tilde\mu_{\bar r_v},\tilde\mu'_{\bar r_v}) + |\bar X_v - \bar X'_v|^2 d\bar r_v\Bigr]\\
        &\leq 8 C_{\bar r} C_\sigma^2 \int_0^{u_*} \E^{\bar\P}\Bigl[\sup_{v\in [0,u]} |\bar X_v - \bar X'_v|^2 \Bigr] du + 8 C_\sigma^2 \int_0^T \c W_2^2(\tilde\mu_t,\tilde\mu'_t) dt.
    \end{align}
    Finally, for the third term $I_3$,
    \begin{align}
        I_3
        &\leq \E^{\bar\P}\Bigl[\sup_{u\in [0,u_*]} \Bigl|\int_0^u \gamma(\bar r_v,\bar X_v,\bar\xi_v) - \gamma(\bar r_v,\bar X'_v,\bar\xi_v) d\bar\xi_v \Bigr|^2\Bigr]\\
        &\leq C_\gamma^2 C_{\bar\xi}^2 \E^{\bar\P}\Bigl[\int_0^{u_*} |\bar X_v - \bar X'_v|^2 dv\Bigr]
        \leq C_\gamma^2 C_{\bar\xi}^2 \int_0^{u_*} \E^{\bar\P}\Bigl[\sup_{v\in [0,u]} |\bar X_v - \bar X'_v|^2 \Bigr] du.
    \end{align}
    In total, we obtain for all $u_*\in [0,1]$,
    \begin{align}
        &\E^{\bar\P}\Bigl[\sup_{u\in [0,u_*]} |\bar X_u - \bar X'_u|^2\Bigr]\\
        &\leq 4(2C_{\bar r}^2 C_b^2 + 8C_{\bar r} C_\sigma^2 + C_\gamma^2 C_{\bar\xi}^2) \int_0^{u_*} \E^{\bar\P}\Bigl[\sup_{v\in [0,u]} |\bar X_v - \bar X'_v|^2 \Bigr]
        + 4(2C_{\bar r} C_b^2 + 8C_\sigma^2) \int_0^T \c W_2^2(\tilde\mu_t,\tilde\mu'_t) dt.
    \end{align}
    Using Gronwall's Lemma, this leads us to
    \begin{align}
        \E^{\bar\P}\Bigl[\sup_{u\in [0,1]} |\bar X_u - \bar X'_u|^2\Bigr] 
        &\leq (8C_{\bar r} C_b^2 + 32C_\sigma^2) \int_0^T \c W_2^2(\tilde\mu_t,\tilde\mu'_t) dt \exp\Bigl(8C_{\bar r}^2 C_b^2 + 32C_{\bar r} C_\sigma^2 + 4C_\gamma^2 C_{\bar\xi}^2 \Bigr).
    \end{align}
\end{proof}

\let\c\predefinedC
\printbibliography

\end{document}